\documentclass[final,onefignum,onetabnum]{siamart190516}
\usepackage{graphicx}
\usepackage{subfig}
\usepackage{url}
\usepackage{enumerate}
\usepackage{amsmath, amssymb}
\usepackage{bbm}
\usepackage{mathrsfs}
\usepackage{booktabs,amsfonts,dcolumn}
\numberwithin{equation}{section}
\newtheorem{prop}{Proposition}

\newtheorem{remark}{Remark}
\usepackage{xcolor}
\newcommand{\diver}{\nabla\cdot}
\newcommand{\p}{\partial}
\newcommand{\curl}{\nabla\times}
\renewcommand{\diver}{\nabla\cdot}
\newcommand{\R}{\mathbb{R}}
\renewcommand{\P}{\mathbb{P}}
\newcommand{\Ha}{H_a}
\renewcommand{\Re}{\mathrm{Re}}
\newcommand{\Rm}{\mathrm{Re}_m}
\renewcommand{\d}{\,\mathrm{d}}
\renewcommand{\epsilon}{\varepsilon}
\renewcommand{\vec}{\boldsymbol}
\newcommand{\fancy}[1]{\mathcal{#1}}
\newcommand{\vphi}{\vec{\phi}}
\newcommand{\vbeta}{\vec{\beta}}

\newcommand{\vmag}{\vec{b}}
\newcommand{\vmtestC}{\vec{c}}
\newcommand{\vmtestD}{\vec{d}}
\newcommand{\EPspace}{\mathscr{P}}
\newcommand{\EPspaceh}{\mathscr{P}_h}
\newcommand{\embed }{\gamma}
\newcommand{\vfancy}[1]{\vec{\fancy{#1}}}
\newcommand{\ofancy}[1]{\overline{\fancy{#1}}}
\newcommand{\ovfancy}[1]{\overline{\vec{\fancy{#1}}}}

\title{An \textit{a posteriori} error analysis for the equations of stationary
incompressible magnetohydrodynamics
\thanks{Submitted to the editors June 5, 2020\@. \funding{J. H. Chaudhry's
work is supported by the NSF-DMS 1720402. The work of A. E. Rappaport and J. N.
Shadid was partially supported by the U.S. Department of Energy, Office of
Science, Office of Advanced Scientific Computing Research, Applied Mathematics
Program and by the U.S. Department of Energy, Office of Science, Office of
Advanced Scientific Computing Research and Office of Fusion Energy Sciences,
Scientific Discovery through Advanced Computing (SciDAC) program. Sandia
National Laboratories is a multimission laboratory managed and operated by
National Technology and Engineering Solutions of Sandia, LLC, a wholly owned
subsidiary of Honeywell International, Inc., for the U.S. Department of
Energy’s National Nuclear Security Administration under contract DE-NA0003525.
This paper describes objective technical results and analysis. Any subjective
views or opinions that might be expressed in the paper do not necessarily
represent the views of the U.S. Department of Energy or the United States
Government.}} }

\headers{\textit{A posteriori} analysis of MHD}
{J.~H.~Chaudhry, A.~E.~Rappaport, and J.~N.~Shadid}

\author{Jehanzeb H. Chaudhry\thanks{Department of Mathematics and Statistics, University of New Mexico (\email{jehanzeb@unm.edu}, \url{https://math.unm.edu/\~jehanzeb}).}
\and Ari E. Rappaport\thanks{Department of Mathematics and Statistics, University of New Mexico
and Center for Computing Research, Sandia National Laboratories, Albuquerque NM
 (\email{aerappa@sandia.gov}).}
\and John N. Shadid\thanks{Center for Computing Research, Sandia National Laboratories, Albuquerque NM and Department of Mathematics and Statistics, University of New Mexico (\email{jnshadi@sandia.gov}).}
}
\begin{document}
\maketitle

\begin{abstract}
Resistive magnetohydrodynamics (MHD) is a continuum base-level model for
conducting fluids (e.g. plasmas and liquid metals) subject to external magnetic
fields. The efficient and robust solution of the MHD system poses many
challenges due to the strongly nonlinear, non self-adjoint, and highly coupled
nature of the physics.  In this article, we develop a robust and accurate
\textit{a posteriori} error estimate for the numerical solution of the
resistive MHD equations based on the exact penalty method.  The error estimate
also isolates particular contributions to the error in a quantity of interest
(QoI) to inform discretization choices to arrive at accurate solutions.  The
tools required for these estimates involve duality arguments and computable
residuals.
\end{abstract}



\section{Introduction}
The resistive magnetohydrodynamics (MHD) equations provide a continuum model for conducting
fluids subject to magnetic fields and are often used to model important
applications e.g. higher-density, highly collisional plasmas.  In this context, MHD
calculations aid physicists in understanding both thermonuclear fusion and
astrophysical plasmas  as well as understanding the behavior of liquid metals
\cite{goedbloed_poedts_mills_romaine,muller_ulrich_buhler_leo_2010}.  
From a phenomenological perspective, the  
governing equations of MHD couple Navier-Stokes equations for fluid dynamics
with a reduced set of Maxwell's equations for low frequency electromagnetic
phenomenon. Structurally, the equations of MHD form a highly coupled, nonlinear,
non self-adjoint system of partial differential equations (PDEs). Analytical solutions to
the MHD system cannot be obtained for practical configurations; instead
numerical solutions are sought. 
The theoretical and numerical analysis of MHD dates back to the pioneering work
of Temam \cite{Temam2001}.  Finite element formulations of incompressible
resistive MHD   include stabilization methods based on variational multiscale
(VMS) approaches \cite{Lin2018, Lin2019, Trelles2014}, exact and weighted
penalty methods \cite{gunzburger_meir_peterson_1991, gerbeau_2000,
phillips_elman_cyr_shadid_pawlowski_2014, martin_dauge_2002}, first order
system least squares (FOSLS) \cite{Adler2010, Adler-Nested2010, Adler2013,
Hsieh2009} and structure preserving methods \cite{nedelec_1980,Evans1988,
Hyman1997, bochev2001matching, Miller2019}. A survey of various numerical
techniques for MHD is found in \cite{Gerbeau2006}.
In this article we restrict ourselves to the stationary MHD equations based on
the exact penalty finite element formulation, originally developed
in~\cite{gunzburger_meir_peterson_1991} from a finite element method
discretization. We do not employ specialized solver strategies e.g. block
preconditioning as the problem size we consider does not merit it.

The numerical solution of complex equations like the MHD equations often have
a significant discretization error for solution with significant fine scale
spatial structures.  This error must be quantified for the reliable use of MHD
equations in numerous science and engineering fields.  Accurate error
estimation is a key component of predictive computational science and
uncertainty quantification \cite{Estep2008,EMT09-1,JBE18}. Moreover, the  error
depends on a complex interaction between many contributions.  Thus, the
availability of an accurate error estimate and the different sources of error
also offers the potential of optimizing the choice of discretization parameters
in order to achieve desired accuracy in an efficient fashion. In this work we
leverage adjoint based \textit{a posteriori} error estimates for a quantity of
interest (QoI) related to to the solution of the MHD equations.  These
estimates  provide a concrete error analysis of different contributions of
error, as well as inform solver and discretization strategies.

In many scientific and engineering applications, the goal of running
a simulation is to compute a set of specific QoIs of the solution, for example
the drag over a plane wing in the context of the compressible Navier-Stokes
equations.  Adjoint based analysis~\cite{giles_suli, becker_rannacher_2003,
Estep1995, eehj_book_96, AO2000, BangerthRannacher} for quantifying the error
in a numerically computed QoI has found success for a wide variety of numerical
methods and discretizations ranging from finite
element~\cite{chaudhry_2018,Estep2008,Estep2010,Chaudhry2019},
finite volume~\cite{barth04}, time integration~\cite{Estep1995, CET+2016,
CEG+2015, Chaudhry2017}, operator splitting techniques~\cite{Estep2008,
Estep2010} and uncertainty quantification~\cite{EstepUQ1,EstepUQ2,JBE18}.

Adjoint based \textit{a posteriori} error analysis uses
variational analysis and duality to relate errors to computable residuals. In
particular, one solves an adjoint problem whose solution provides the
residual weighting to produce the error in the QoI. The technique also
naturally allows to identify and isolate different components of error arising
from different aspects of discretization and solution methods, by analyzing
different components of the weighted residual separately.

This article carries out the first adjoint based \textit{a posteriori} error
analysis for the MHD equations to the best of our knowledge.  The definition of
the adjoint operator to the strong form of the MHD system is not obvious since
that system is rectangular, and hence the weak form of the exact penalty method
is needed for forming the appropriate adjoint problem. We further provide
theory supporting the well-posedness of the adjoint weak form. Additionally,
the resulting \textit{a posteriori} error estimate is  decomposed to identify
various sources of error, and the efficacy  of the error estimate is
demonstrated on a set of benchmark MHD problems.

The remainder of the article is organized as follows. In \S\ref{sect:primal},
we review the equations of incompressible resistive MHD,  present the exact
penalty weak form and the finite element method to numerically solve the
problem. In \S\ref{sect:error-analysis-abs} we develop  theoretical results for
adjoint based \textit{a posteriori} error analysis for an abstract problem
representative of the exact penalty weak form. We apply these results to the
MHD equations in \S \ref{sect:MHD-specific-case} to develop an \textit{a
posteriori} error estimate.  In \S\ref{sect:numerical-results} we present
numerical results to demonstrate the accuracy and utility of the error
estimates produced by our method. In \S\ref{sect:theoretical-results} we give
details of the derivation of the nonlinear operators in the weak adjoint form
as well as a well-posedness argument for the adjoint problem.

\section{Exact penalty formulation and discretization}\label{sect:primal}
In this section we describe the nondimensionalized  equations of incompressible
stationary MHD, a stabilized weak form of the MHD system and a finite element
method for its solution.

\subsection{The MHD equations}
Throughout the rest of the paper, let $\Omega\subset\R^d$, $d = $2 or $3$ be
a bounded, convex polyhedral domain with boundary $\partial \Omega$. The
assumptions on the domain are necessary for the solution strategy we
choose, as elaborated  in \S\ref{sect:exact-penalty}.  The
nondimensional equations for stationary incompressible MHD in $\Omega$ are
given by
\begin{subequations}
\begin{align}
-\frac{1}{\Re}\Delta\vec{u} + (\vec{u}\cdot\nabla)\vec{u} + \nabla p -\kappa(\curl\vmag)\times\vmag&=\vec{f},\label{eq:MHD-momentum}\\
\diver\vec{u} &= 0,\\
\frac{\kappa}{\Rm}\curl(\curl\vmag) - \kappa\curl(\vec{u}\times\vmag) &= \vec{0},\label{eq:MHD-induction}\\
\diver\vmag &= 0,\label{eq:MHD-divB}
\end{align}\label{eq:MHD-system}
\end{subequations}
where the unknowns are the velocity $\vec{u}$, the magnetic field $\vmag$, and
the pressure $p$.  The nondimensional parameters are the fluid Reynolds number
$\Re>0$, Magnetic Reynolds number $\Rm>0$, and interaction parameter
$\kappa=\Ha^2/(\Re\Rm)$, where $\Ha>0$ is the Hartmann number. We require the
source term $\vec{f}\in\vec{H}^{-1}(\Omega)$. For $x \in \Omega$ we have
$\vec{u}(x) \in \mathbb{R}^d$, $\vmag(x) \in \mathbb{R}^d$, $p(x) \in
\mathbb{R}$ and $\vec{f}(x) \in \mathbb{R}^d$.  We supplement the system
\eqref{eq:MHD-system} with boundary conditions,
\begin{subequations}
\begin{align}
\vec{u}&=\vec{g},&&\text{on }\p\Omega,\\
\vmag\times\vec{n}&=\vec{q}\times\vec{n},&&\text{on }\p\Omega.
\end{align}\label{eq:MHD-BCs}
\end{subequations}

Referring to \eqref{eq:MHD-system}, we observe there are $2d+2$ and only $2d+1$
unknowns \cite{phillips_elman_cyr_shadid_pawlowski_2014}.  Effectively
enforcing the solenoidal constraint \eqref{eq:MHD-divB} (an involution of the
transient MHD system) is an active area of research.  Techniques include
compatible discretizations \cite{schotzau_2004,bochev2001matching}, vector
potential \cite{Adler2019, Shadid2010} and divergence cleaning
\cite{dedner_kemm_kroner_munz_schnitzer_wesenberg_2002, Kuzmin2020} as well as
the exact penalty method \cite{gunzburger_meir_peterson_1991, gerbeau_2000,
phillips_elman_cyr_shadid_pawlowski_2014}.  In this article, we consider the
exact penalty method which we further describe in \S\ref{sect:exact-penalty}.

\subsection{Function spaces for the MHD system}
We make use of the standard spaces $L^2(\Omega)$ and $H^m(\Omega)$ as well as
their vector counterparts $\vec{L}^2(\Omega)$ and $\vec{H}^m(\Omega)$. The
$L^2(\Omega)$ (or $\vec{L}^2(\Omega)$) inner product is denoted by $(\cdot,
\cdot)$ and the norm is denoted by $\| \cdot \|$, while the $H^1(\Omega)$ (or
$\vec{H}^1(\Omega)$) norm is denoted by $\| \cdot \|_1$.  The norm in
$\mathbb{R}^d$ is denoted by $\| \cdot \|_{\mathbb{R}^d}$. The details of these
function spaces are given in  Appendix \ref{app:func_spaces}. Further useful
relations used throughout the text are given in  Appendix \ref{append:vect-ind}
and Appendix \ref{app:useful_ineqs}.  For $\vmag\in \vec{H}^1(\Omega)$, we
define $\nabla\vmag:=\begin{bmatrix}\nabla b_1,\dots,\nabla b_d\end{bmatrix}^T$
as a matrix whose rows are the gradients of the components of $\vmag$.  The
relevant subspaces of $\vec{H}^1(\Omega)$ needed to satisfy the boundary
conditions (in the sense of the trace operator) are,
\begin{align}
&\vec{H}_0^1(\Omega):=\{\vec{w}\in\vec{H}^1(\Omega):\vec{w}|_{\p\Omega}\equiv\vec{0}\}\label{eq:H-zero-one},\\
&\vec{H}_\tau^1(\Omega):=\{\vec{w}\in\vec{H}^1(\Omega):(\vec{w}\times\vec{n})|_{\p\Omega}\equiv\vec{0}\}\label{eq:H-tau-one}.
\end{align}
Finally, we define the product space,
\begin{align}
\EPspace&:= \vec{H}^1_0(\Omega)\times\vec{H}^1_\tau(\Omega)\times L^2(\Omega).
\end{align}
We also remark that for $d=2$, we use the natural inclusion of
$\R^2\hookrightarrow\R^3$, $\begin{bmatrix}v_1, v_2\end{bmatrix}^T \mapsto
\begin{bmatrix}v_1, v_2, 0\end{bmatrix}^T$ to define the operators $\curl$ and
$\times$. Thus for $\vec{v},\vec{w}\in\vec{H}^1(\Omega)$, we have that
\begin{align*}
&\curl \vec{v} = \left(\frac{\p v_y}{\p x} - \frac{\p v_x}{\p y}\right)\hat{\vec{k}},\quad
\vec{v}\times\vec{w} = \left(v_xw_y - v_yw_x\right)\hat{\vec{k}}.
\end{align*}

\subsection{Exact penalty formulation}\label{sect:exact-penalty}
In this section we present the weak form of the stationary incompressible MHD
system based on the exact penalty formulation
\cite{gunzburger_meir_peterson_1991}.  The exact penalty method requires that
the domain $\Omega$ is bounded, convex and polyhedral. This ensures that
$\vec{H}(\textbf{curl},\Omega)\cap \vec{H}(\mathrm{div},\Omega)$ is
continuously embedded in $\vec{H}^1(\Omega)$ \cite{nedelec_1980, Gerbeau2006}.
We also assume homogeneous Dirichlet boundary conditions i.e.
$\vec{g}=\vec{q}=\vec{0}$. Non-homogeneous boundary conditions can be dealt
with through standard lifting arguments as discussed in
\S\ref{sec:boundary_considerations}.
The exact penalty weak problem corresponding to \eqref{eq:MHD-system} and
\eqref{eq:MHD-BCs} is: find $U=(\vec{u},\vmag,p)\in\EPspace$ such that
\begin{equation}
\mathcal{N}_{EP}(U,V) = (\vec{f},\vec{v}),\quad\forall\,V\in\EPspace,\label{eq:exact-penalty-primal}
\end{equation}
where the nonlinear form $\mathcal{N}_{EP}$ is defined for all $V=(\vec{v},
\vmtestC, q)\in\EPspace$ by
\begin{equation}
\begin{aligned}
\mathcal{N}_{EP}(U,V) :&=\frac{1}{\Re}(\nabla\vec{u},\nabla\vec{v})+
(\vfancy{C}(\vec{u}),\vec{v})- (p,\diver\vec{v}) + (q,\diver\vec{u})\\
&-\kappa(\vfancy{Y}(\vmag),\vec{v})
- \kappa(\vfancy{Z}(\vec{u},\vmag),\vmtestC)\\
&+\frac{\kappa}{\Rm}(\curl\vmag,\curl\vmtestC)
+\frac{\kappa}{\Rm}(\diver\vmag,\diver\vmtestC),
\end{aligned}\label{eq:EP-weak-form}
\end{equation}
and the nonlinear operators are defined by
\begin{subequations}
\begin{align}
\vfancy{C}(\vec{u}) &:= (\vec{u}\cdot\nabla)\vec{u},\\
\vfancy{Y}(\vmag) &:= (\curl\vmag)\times\vmag,\\
\vfancy{Z}(\vec{u},\vmag) &:=\curl(\vec{u}\times\vmag).
\end{align}
\label{eq:all-nonlin-ops}
\end{subequations}
All except the last term in the weak form arise from multiplying
\eqref{eq:MHD-momentum}-\eqref{eq:MHD-induction} by test functions and
performing integration by parts. The last term,
$\frac{\kappa}{\Rm}(\diver\vmag,\diver\vmtestC)$, effectively enforces the
solenoidal involution \eqref{eq:MHD-divB} since, assuming the aforementioned
restrictions on the domain, there exists a function (see
\cite{gunzburger_meir_peterson_1991, Girault1986}) $b_0\in H^2(\Omega)$ such
that
\begin{equation}
\diver\nabla b_0 = \diver\vmag,\text{ and }\nabla b_0\in\vec{H}_\tau^1(\Omega).
\end{equation}
Thus, we choose $V=(\vec{0},\nabla b_0, 0)$ in \eqref{eq:EP-weak-form} and use
\eqref{div-cross-int} so that \eqref{eq:exact-penalty-primal} reduces to
\begin{equation}
(\diver\vmag,\diver\nabla b_0) = (\diver\vmag,\diver\vmag) = 0,
\end{equation}
and hence \eqref{eq:MHD-divB} is satisfied almost everywhere in $\Omega$.
\begin{remark}
  \label{rem:wellposed}
  The existence of the solution to the problem \eqref{eq:exact-penalty-primal}
  is proven in \cite[Theorem 4.6]{gunzburger_meir_peterson_1991} as well as in
  \cite[Theorem 3.22]{Gerbeau2006}, while uniqueness is proven in \cite[Theorem
  4.7]{gunzburger_meir_peterson_1991} and also in \cite[Theorem
    3.22]{Gerbeau2006}. Both uniqueness proofs rely on a ``small data''
    assumption, i.e. inequalities bounding the nondimensionalised constants,
    $\Re,\Rm$ and $\kappa$, in terms of the data $\vec{f},\vec{g}$ and
    $\vec{q}$.
\end{remark}

\subsection{Finite element method}\label{sect:FEM}
We introduce the standard continuous Lagrange finite element spaces. Let
$\mathcal{T}_h$ be a simplicial decomposition of $\Omega$, where $h$ denotes
the maximum diameter of the elements of $\mathcal{T}_h$, such that the union of
the elements of $\mathcal{T}_h$ is $\Omega$, and the intersection of any two
elements is either a common edge, node, or is empty.
The standard Lagrange space finite element space of order $q$
is then
\begin{equation} \P_h^q:=\big\{v \in C(\Omega):\forall
K\in\mathcal{T}_h,\,v|_K\in\P^q(K)\big\}, \end{equation}
where $\P^q(K)$ is the space of polynomials of degree at most $q$ defined on
the element $K$. Additionally, our finite element space satisfies the
Ladyzhenskaya-Babu\v{s}ka-Brezzi condition stability
condition~\cite{boffi2013mixed} for the velocity pressure pair, e.g.
$\EPspaceh=\P_h^2(\Omega)\times \P_h^1(\Omega)\times \P_h^1(\Omega)$.  Then the
discrete problem to find an approximate solution $U_h=(\vec{u}_h, \vmag_h,
p_h)\in\EPspaceh$ to \eqref{eq:EP-weak-form} is,
\begin{equation}
\mathcal{N}_{EP}(U_h,V_h)=(\vec{f},\vec{v}_h)\quad \forall\,V_h\in\EPspaceh.\label{eq:primal-FE}
\end{equation}
Note there is no restriction on the finite element space for $\vmag_h$,
which is an advantage of this method. 
The existence and uniqueness of the solution of the discrete problem
\eqref{eq:primal-FE} is also demonstrated in Gunzburger et al.
\cite{gunzburger_meir_peterson_1991} with the same assumptions of the data as
discussed in Remark \ref{rem:wellposed}.

\subsection{Quantity of interest (QoI)}

The goal of a numerical simulation is often to compute some functional of the
solution, that is, the QoI.  In particular, QoIs considered in this article
have the generic form,
\begin{equation}
\text{QoI} = \int_{\Omega}\Psi\cdot U\d x = (\Psi, U) \label{eq:QoI}
\end{equation}
where $U$ is defined by \eqref{eq:exact-penalty-primal} and $\Psi \in
\vec{L}^2(\Omega) \times \vec{L}^2(\Omega) \times L^2(\Omega) \equiv
[L^2(\Omega)]^{2d+1}$.  For example in two dimensions, to compute the average
of the $y$ component of velocity $u_y$ over a region $\Omega_c\subset\Omega$,
set
$\Psi=\frac{1}{|\Omega_c|}\begin{bmatrix}0, \mathbbm{1}_{\Omega_c},0 , 0,
0)\end{bmatrix}^T$, where $\mathbbm{1}_S$ denotes the characteristic function
over a set $S$.  In the examples presented later, the QoIs physically represent
quantities representative of the average flow rate, or the average induced
magnetic field. We seek to compute error estimates in the QoI using duality
arguments as presented in the following subsection.

\section{Abstract \textit{a posteriori} error analysis}\label{sect:error-analysis-abs}

In this section we consider an abstract variational setting for \textit{a
posteriori} analysis based on the ideas from
\cite{Estep1995,Eriksson1995,giles_suli,AO2000, BangerthRannacher}.  Let
$\mathscr{W}$ be a Hilbert space with inner-product $\langle\cdot,\cdot\rangle$
and let $\mathscr{V}$ be a dense subspace of $\mathscr{W}$.  Throughout this
section $u \in \mathscr{V}$ refers to the solution of an abstract variational
problem (e.g. solution of \eqref{eq:bilinvar} or
\eqref{eq:generic-nonlin-problem}). An example of such a variational problem
is the exact penalty problem as described in \S\ref{sect:exact-penalty}.
Moreover, we denote $u_h\in \mathscr{V}_h$ as a numerical
approximation to $u$, where  $\mathscr{V}_h$  is a finite dimensional subspace
of $\mathscr{V}$, and denote the error as $e = u - u_h$. Finally, $w$ and $v$
refer to arbitrary functions, and their spaces are made clear when we use these
functions. For the QoI, consider bounded linear functionals of the form,
\begin{equation}
  Q(w) = \langle \psi, w \rangle, \quad \forall w \in \mathscr{W},
  \label{eq:abs-QoI_blf}
\end{equation}
for some fixed $\psi \in \mathscr{W}$. The QoI is then,
\begin{equation}
 Q(u) = \langle \psi, u \rangle.
 \label{eq:abs-QoI}
 \end{equation}

For example, in \eqref{eq:QoI}, $\langle \psi,
u \rangle = ( \Psi, U)$, that is the inner-product  is the $L^2$ inner product.
The aim of the \textit{a posteriori} analysis is to compute the error in the
QoI, $ Q(u) - Q(u_h) = \langle \psi, u \rangle - \langle \psi, u_h \rangle
= \langle \psi, e \rangle$. We briefly describe the analysis for linear
problems in \S \ref{sec:apost_lin} and then consider nonlinear problems in \S
\ref{sec:apost_nonlin}.

\subsection{Linear variational problems}
\label{sec:apost_lin}
We consider the problem of evaluating \eqref{eq:abs-QoI} where $u$ is the
solution to the linear variational problem: find $u\in\mathscr{V}$ such that
\begin{equation}
  a(u,v) = \langle f, v\rangle,\quad\forall v\in\mathscr{V},
  \label{eq:bilinvar}
\end{equation}
where $a:\mathscr{V}\times \mathscr{V}\to \R$ is a bilinear form. We then
define the adjoint bilinear form $a^*:\mathscr{V}\times \mathscr{V} \to\R$ as
the unique bilinear form satisfying
\begin{equation}
   a^*(w,v) = a(v,w),\quad\forall w,v\in \mathscr{V},\label{eq:varadj}
\end{equation}
see \cite{giles_suli, becker_rannacher_2003}.
If $\phi$ solves the dual problem: find $\phi\in \mathscr{V}$ such that
\begin{equation}
   a^*(\phi,v) = \langle \psi, v\rangle,\quad\forall v\in \mathscr{V},\label{eq:lin-adj-defn}
\end{equation}
then we have the following error representation.
\begin{theorem}
   The error in the QoI \eqref{eq:abs-QoI} is represented as $\langle \psi,
   e\rangle  = \langle f,\phi\rangle - a(u_h, \phi)$, where $u$ is the
   solution to  \eqref{eq:bilinvar}, $u_h $ is a numerical approximation, $e
   = u - u_h$ and $\phi$ is the solution to
  \eqref{eq:lin-adj-defn}.
\end{theorem}
\begin{proof}
  The proof is a straightforward computation,
  \begin{equation}
  \begin{aligned}
  \langle \psi, e\rangle = a^*(\phi, e)&=a(e, \phi)= a(u,\phi) - a(u_h, \phi)= \langle f,\phi\rangle - a(u_h, \phi).
  \end{aligned}\label{eq:abstract-err-rep}
  \end{equation}
\end{proof}
Note from  the proof above that a simple yet important property of the
adjoint bilinear form $a^*(\cdot,\cdot)$ is,
\begin{equation}
   \label{eq:key_lin_proof}
   a^*(v, e)= a(u,v) - a(u_h, v),
\end{equation}
for $w \in \mathscr{V}$.  We will use this property in motivation the analysis
for nonlinear problems in \S\ref{sec:apost_nonlin}.
\subsection{Nonlinear variational problems}
\label{sec:apost_nonlin}
Again, our goal is to evaluate \eqref{eq:abs-QoI} where now $u$ is the solution
to the \textit{nonlinear} variational problem: find $u$ in $\mathscr{V}$ such
that 
\begin{equation}
  \fancy{N}(u,v) = \langle f, v\rangle,\quad\forall v\in \mathscr{V},
  \label{eq:generic-nonlin-problem}
\end{equation} 
and $\fancy{N}:\mathscr{V}\times \mathscr{V}\to \R$ is linear in the second
argument but may be nonlinear in the first argument.  There is no
straightforward definition of an adjoint operator corresponding to  a nonlinear
problem.  However, a common choice useful for various kinds of analysis  is
based on linearization~\cite{MAS1996,Marchuk1995,Chaudhry2019, Chaudhry2017,
chaudhry_2018, Estep2010}. This choice enables the definition of an adjoint
bilinear form $\ofancy{N}^*(\cdot,\cdot)$ which satisfies the useful property,
\begin{equation}
	\label{eq:main_linearized_prop}
	 \ofancy{N}^*(v,e) =  \fancy{N}(u,v) - \fancy{N}(u_h,v),
\end{equation}
for all $v \in \mathscr{V}$. This property is inspired by
\eqref{eq:key_lin_proof}.

We now present a specific case of this analysis such the problem
\eqref{eq:generic-nonlin-problem}  mimics the setup of the exact penalty
problem in \eqref{eq:exact-penalty-primal}. Let
$\mathscr{V}=\prod_{i=1}^n\mathscr{V}_i$ and
$\mathscr{W}=\prod_{i=1}^n\mathscr{W}_i$ be product spaces of Hilbert spaces
such that $\mathscr{V}_i$ is a dense subspace of $\mathscr{W}_i$ for each $i$.
The left hand side in problem \eqref{eq:generic-nonlin-problem} is now more
specifically given by
\begin{equation} \label{eq:sum_nonlin_forms}
\fancy{N}(v,w) = \sum_{i=1}^m\langle N_i(v),w_{\ell_i}\rangle + a(v, w),
\end{equation}
where $a(\cdot,\cdot)$ is a bilinear form,  $\ell_i\in\{1,\dots,n\}$ and $N_i:
\mathscr{V} \rightarrow \mathscr{W}_{\ell_i}$ are nonlinear operators.  For
a  solution/approximation pair ($u/u_h$) to \eqref{eq:generic-nonlin-problem},
define the matrix $\ofancy{J}$, where each entry
$\ofancy{J}_{ij}:\mathscr{V}_j\to\mathscr{W}_{\ell_i}$ is given by
\begin{align}
\ofancy{J}_{ij}v_j = \int_0^1\frac{\p N_i}{\p u_j}(su+(1-s)u_h)\d s\,v_j,\label{eq:averaged-jacobian}
\end{align}
where $v_j \in \mathscr{V}_j$ and  $\frac{\p N_i}{\p u_j}(\cdot)$ denotes the
partial derivative of $N_i$ with respect to the argument $u_j$. Define the
linearized operator 
$\bar{N}_i:\mathscr{V}\to\mathscr{W}_{\ell_i}$ by
\begin{equation}
\begin{aligned}
\bar{N}_iv&=\int_0^1\frac{\p N_i}{\p u}(su+(1-s)u_h)\d s\cdot\,v\\
&=\sum_{j=1}^n\int_0^1\frac{\p N_i}{\p u_j}(su+(1-s)u_h)\d s\,v_j
=\sum_{j=1}^n\ofancy{J}_{ij}v_j,
\end{aligned}\label{eq:Nibar}
\end{equation}
for $v \in \mathscr{V}$.
Now since each $\bar{N}_i$ is linear, we may
define the bilinear forms, $\ofancy{\nu}_i:\mathscr{V}\times\mathscr{V}\to\R$,
by
\begin{align}
  \label{eq:nui-defn}
  \ofancy{\nu}_i(v,w) &= \langle\bar{N}_iv,w_{\ell_i}\rangle
  =\left\langle\sum_{j=1}^n\ofancy{J}_{ij}v_j,w_{\ell_i}\right\rangle
  =\sum_{j=1}^n\left\langle\ofancy{J}_{ij}v_j,w_{\ell_i}\right\rangle,
\end{align}
for $v,w \in \mathscr{V}$.
Define $\ofancy{\nu}_i^*(v,w) = \ofancy{\nu}_i(w,v)$, and adjoint operators $\ofancy{J}_{ij}^*$ to $\ofancy{J}_{ij}$ satisfying 
\begin{equation}
\label{eq:basic_adoint_id}
\langle \ofancy{J}_{ij} w, v \rangle = \langle w , \ofancy{J}_{ij}^* v\rangle
\end{equation}
for $w \in \mathscr{V}_j$ and $v \in \mathscr{V}_{\ell_i}$. Hence, we can also write using the definition \eqref{eq:nui-defn},
\[\ofancy{\nu}^*_i(v,w) = \sum_{j=1}^n\langle w_j, \ofancy{J}_{ij}^*v_{\ell_i}\rangle.\]
for $v,w \in \mathscr{V}$. Also since $a(\cdot,\cdot)$ in
\eqref{eq:sum_nonlin_forms} is a bilinear form, we have from the definition
\eqref{eq:varadj} that $a^*(w,v) = a(v,w)$ for $v,w \in \mathscr{V}$.
With these definitions in mind, we further define a composite adjoint bilinear
form, $\ofancy{N}^*:\mathscr{V}\times\mathscr{V}\to\R$, as
\begin{equation}
\label{eq:prod_adj_weak_form}
\begin{aligned}
\ofancy{N}^*(v,w) &= \sum_{i=1}	^m\ofancy{\nu}^*_i(v,w) + a^*(v,w)
=\sum_{i=1}^m\sum_{j=1}^n\langle
w_j,\ofancy{J}_{ij}^*v_{\ell_i}\rangle+a^*(v,w),
\end{aligned}
\end{equation}
for $u,v \in \mathscr{V}$. 
Then if $\phi \in \mathscr{V}$ solves the dual problem,
\begin{equation}
\label{eq:adj_prod_spaces}
\ofancy{N}^*(\phi, v) = \langle\psi, v\rangle,\,\forall v\in \mathscr{V},
\end{equation}
we then have the following abstract error representation.
\begin{theorem}\label{thm:rep}
 The error in the QoI \eqref{eq:abs-QoI} is represented as $\langle\psi,
 e\rangle =\langle f,\phi\rangle -\mathcal{N}(u_h, \phi)$ where $u$ is the
 solution to \eqref{eq:generic-nonlin-problem}, $u_h$ is a numerical
 approximation of $u$, $e=u-u_h$, and $\phi$ is the solution to
 \eqref{eq:adj_prod_spaces}.
\end{theorem}
\begin{proof}
We compute, starting by replacing $v$ by $e$ in \eqref{eq:adj_prod_spaces},
\begin{align*}
\langle \psi, e\rangle 
&= \ofancy{N}^*(\phi, e) = \sum_{i=1}^m\sum_{j=1}^n\langle e_j,\ofancy{J}_{ij}^*\phi_{\ell_i}\rangle + a^*(\phi,e)\\
&=\sum_{i=1}^m\sum_{j=1}^n\langle\ofancy{J}_{ij}e_j,\phi_{\ell_i}\rangle
                       + a(e,\phi)\\
&=\sum_{i=1}^m\langle\overline{N}_ie,\phi_{\ell_i}\rangle + a(e,\phi)\\
&= \sum_{i=1}^m\langle N_i(u) - N_i(u_h),\phi_{\ell_i}\rangle + a(u,\phi) - a(u_h, \phi)\\
&= \sum_{i=1}^m\langle N_i(u) ,\phi_{\ell_i}\rangle + a(u,\phi) -\sum_{i=1}^m \langle N_i(u_h),\phi_{\ell_i}\rangle - a(u_h, \phi)\\
&=\fancy{N}(u,\phi) - \fancy{N}(u_h,\phi) = \langle f,\phi \rangle - \fancy{N}(u_h,\phi).
\end{align*}
\end{proof}
The main result of this theorem is that computing the adjoint to a nonlinear
form is reduced to computing the adjoint for the averaged entries,
$\ofancy{J}_{ij}$. 
 
\section{\textit{A posteriori} error estimate for the MHD equations}
\label{sect:MHD-specific-case}
The analysis in \S \ref{sec:apost_nonlin} applies directly to the MHD
equations.  The inner product $\langle \cdot, \cdot \rangle$ of the last
section is represented by the $[L^2(\Omega)]^{2d+1}$ inner product
$(\cdot,\cdot)$.  The  linear and nonlinear terms in the exact penalty weak
form \eqref{eq:exact-penalty-primal} are mapped to match
\eqref{eq:sum_nonlin_forms}.
The  mapping between the abstract formulation and MHD equation is shown in
Table~\ref{tab:mapping_abstract_to_mhd}.

\begin{table}[!ht]
   \centering
   \subfloat[\label{tab:not_a}]{
     \small
     \centering
     \begin{tabular}{c|c}\hline
     		Abstract & MHD\\					\hline
       	$\langle, \rangle$ & $(,)$\\
        $m$ & $3$\\
        $\fancy{N}$ & $\fancy{N}_{EP}$\\
        $u$ & $U$\\
        $v$ & $V$\\        
        $N_i$ & $N_{EP,i}$ \\ \hline        
     \end{tabular}
   }
   \hfill
   \subfloat[\label{tab:not_b}]{
     \small
     \centering
     \begin{tabular}{c|c}\hline
     Abstract & MHD\\					\hline
       $\langle f , v \rangle$ & $(\vec{f},\vec{v})$\\ 
        $u_1$ & $U_1 \equiv \vec{u}$\\
        $u_2$ & $U_2 \equiv \vmag$\\
        $u_3$ & $U_3 \equiv p$\\
        $v_1$ & $V_1 \equiv \vec{v}$\\
        $v_2$ & $V_2 \equiv \vmtestC$\\
        \hline
     \end{tabular}
   }
   \hfill
   \subfloat[\label{tab:not_c}]{
     \small
     \centering
     \begin{tabular}{c|c}\hline
     Abstract & MHD\\					\hline       
        $v_3$ & $V_3 \equiv q$\\
        $\ofancy{J}_{11}^*$& $\ovfancy{Z}_{\vec{u}}^*$ \\
        $\ofancy{J}_{12}^*$& $\ovfancy{Z}_{\vmag}^*$ \\
        $\ofancy{J}_{21}^*$& $\ovfancy{Y}^*$ \\
        $\ofancy{J}_{31}^*$& $\ovfancy{C}^*$ \\
        $a$ & $a_{EP}$\\ \hline
     \end{tabular}
   }
   \caption{Mapping between the abstract framework in \S \ref{sect:error-analysis-abs} and the MHD equation in \S \ref{sect:MHD-specific-case}. $\fancy{N}_{EP}$ is given in \eqref{eq:abstract_EP}, $N_{EP,i}$ in \eqref{eq:mhd_N_forms}, $a_{EP}$ in \eqref{eq:mhd_aep} and $\ovfancy{Z}_{\vec{u}}^*, \ovfancy{Z}_{\vmag}^*,\ovfancy{Y}^*,\ovfancy{C}^*$ are given in  \eqref{eq:averaged-entries}.} \label{tab:mapping_abstract_to_mhd}
\end{table}
 For the exact penalty weak form, we have that
\begin{equation}
\label{eq:abstract_EP}
\fancy{N}_{EP}(U,V) = \sum_{i=1}^3(N_{EP,i}(U), V_{\ell_i}) + a_{EP}(U,V),
\end{equation}
where
\begin{equation}
\label{eq:mhd_N_forms}
\begin{aligned}
	(N_{EP,1}(U)\color{black}, V_{2})  &= (\vfancy{Z}(\vec{u},\vmag)\color{black},\vmtestC),\\
	(N_{EP,2}(U)\color{black}, V_{1})  &= (\vfancy{Y}(\vmag)\color{black},\vec{v}),\\
	(N_{EP,3}(U)\color{black}, V_{1})  &= (\vfancy{C}(\vec{u})\color{black}, \vec{v}),
\end{aligned}
\end{equation}
$\vfancy{Z}, \vfancy{Y}, \vfancy{C}$ are in turn defined in \eqref{eq:all-nonlin-ops}, and
\begin{equation}
\label{eq:mhd_aep}	
\begin{aligned}
a_{EP}(U,V) &= \frac{1}{\Re}(\nabla\vec{u},\nabla\vec{v})- (p,\diver\vec{v}) + (q,\diver\vec{u})\\
&+\frac{\kappa}{\Rm}(\curl\vmag,\curl\vmtestC)
+\frac{\kappa}{\Rm}(\diver\vmag,\diver\vmtestC).
\end{aligned}
\end{equation}
The entries
$\ofancy{J}_{11}^*V_2=\ovfancy{Z}_{\vec{u}}^*\vmtestC$,
$\ofancy{J}_{12}^*V_2=\ovfancy{Z}_{\vmag}^*\vmtestC$, 
$\ofancy{J}_{21}^*V_1=\ovfancy{Y}^*\vec{v}$ and $\ofancy{J}_{31}^*V_1=\ovfancy{C}^*\vec{v}$
are,
\begin{equation}
\begin{aligned}
\ovfancy{Z}_{\vec{u}}^*\,\vmtestC&=\tfrac{1}{2}(\vec{u}+\vec{u}_h)\times(\curl\vmtestC),\\
\ovfancy{Z}_{\vmag}^*\,\vmtestC&=-\tfrac{1}{2}(\vmag+\vmag_h)\times(\curl\vmtestC),\\
\ovfancy{Y}^*\vec{v}&=\tfrac{1}{2}\big(-(\curl(\vmag+\vmag_h)\times\vec{v})  +\curl((\vmag+\vmag_h)\times\vec{v})\big),\\
\ovfancy{C}^*\vec{v}&=\tfrac{1}{2}\big((\nabla\vec{u}+\nabla\vec{u}_h)^T\vec{v}
-\left(((\vec{u}+\vec{u}_h)\cdot\nabla)\vec{v}\right)-\big(\diver(\vec{u}+\vec{u}_h)\big)\vec{v},
\end{aligned}
\label{eq:averaged-entries}
\end{equation}
while the remaining $\ofancy{J}_{ij}^*$  entries are zero.
The details of the derivation are given in \S
\ref{sec:adj_deriv_details}.

\subsection{Adjoint problem for incompressible MHD}\label{sect:MHD-weak-adjoint}
We are now prepared to pose a weak adjoint problem corresponding to exact
penalty primal problem \eqref{eq:exact-penalty-primal}.  Based on
\eqref{eq:abstract_EP}, \eqref{eq:averaged-entries} and
\eqref{eq:adj_prod_spaces},
the weak dual problem is therefore be stated as: find $\Phi=(\vphi, \vbeta,
\pi)\in\EPspace$ such that
\begin{equation}
\ofancy{N}_{EP}^*(\Phi,V)=(\Psi,V),\quad\forall\,V=(\vec{v},
\vmtestC, q)\in\EPspace,\label{eq:EP-dual-problem}
\end{equation}
with
\begin{equation}
\begin{aligned}
\ofancy{N}_{EP}^*(\Phi,V)&=\frac{1}{\Re}(\nabla\vphi,\nabla\vec{v})+
\left(\ovfancy{C}^*\vphi, \vec{v}\right)
+(\diver\vec{v}, \pi)- (\diver\vphi, q)\\
&+\frac{\kappa}{\Rm}\left(\curl\vbeta,\curl\vmtestC\right)
+ \frac{\kappa}{\Rm}\left(\diver\vbeta,\diver\vmtestC\right)\\
&-\kappa\left(\ovfancy{Y}^*\vphi,\vmtestC\right)
-\kappa\left(\ovfancy{Z}_{\vec{u}}^*\vbeta,\vec{v}\right)-
\kappa\left(\ovfancy{Z}_{\vmag}^*\vbeta,\vmtestC\right).
\end{aligned}\label{eq:EP-weak-formula}
\end{equation}
Here recall that $\Psi$ is defined by \eqref{eq:QoI}.
The forms of the linear operators $\ovfancy{C}^*,\ovfancy{Y}^*$,
$\ovfancy{Z}_{\vec{u}}^*$ and $\ovfancy{Z}_{\vmag}^*$ are given in
\eqref{eq:averaged-entries}.  We discuss the well-posedness of the adjoint
problem \eqref{eq:EP-dual-problem} in \S \ref{sec:weak_form_adj_well_posedness}.
\subsection{Error representation} \label{sec:err_rep_mhd_homo}

In order to discuss an error representation we need to make the following definition
\begin{definition}
  \label{defn:err}
Define the monolithic error by $E=\begin{bmatrix}\vec{e}_{\vec{u}},
\vec{e}_{\vmag}, e_p\end{bmatrix}^T$ with component errors 
\begin{align}
\vec{e}_{\vec{u}} =\vec{u}-\vec{u}_h,\,\vec{e}_{\vmag} =\vmag-\vmag_h,\,e_p = p-p_h.
\end{align}
where $(\vec{u},\vmag, p)\in\EPspace$ is the solution to 
  \eqref{eq:exact-penalty-primal} and $(\vec{u}_h,\vmag_h,
p_h)\in\EPspace_h$ is the solution to \eqref{eq:primal-FE}.
\end{definition}
We then have the following error representation.
\begin{theorem}[Error representation for exact penalty]\label{thm:err-rep-ep}
The error in the numerical approximation of the QoI \eqref{eq:QoI}
satisfies
\begin{align*}
	(\Psi, E)  = (\vec{f},\vphi) - &\bigg[\frac{1}{\Re}(\nabla\vec{u}_h,\nabla\vphi)+ (\vec{u}_h\cdot\nabla\vec{u}_h,\vphi)\\
	&- (p_h, \diver\vphi)+ \kappa((\curl\vmag_h)\times\vmag_h,\vphi) + (\diver\vec{u}_h,\pi)\\
	&+\frac{\kappa}{{\Rm}}(\curl\vmag_h,\curl\vbeta)+ \kappa(\curl(\vec{u}_h\times\vmag_h),\vbeta)\\
	&+ \frac{\kappa}{\Rm}(\diver\vmag_h,\diver\vbeta)\bigg],
	\end{align*}
where $\Phi=(\vphi,\vbeta,\pi)$ is defined in  \eqref{eq:EP-dual-problem}.
\end{theorem}

\begin{proof}
By Theorem~\ref{thm:rep}, 
	\begin{align*}
    (\Psi, E) &= \ofancy{N}_{EP}^*(\Phi, E) = \fancy{N}_{EP}(U, \Phi) -
    \fancy{N}_{EP}(U_h,\Phi) = (\vec{f},\vphi) - \fancy{N}_{EP}(U_h,\Phi).
	\end{align*}

	\end{proof}

\subsection{Non-homogeneous boundary conditions for the MHD system}
\label{sec:boundary_considerations}

The analysis above easily extends to the case of non-homogeneous boundary
conditions, i.e. when $\vec{g}$ or $\vec{q}$ are not identically zero. First
assume that the numerical solution  $U_h$ the satisfies the non-homogeneous
conditions exactly. That is, $\vec{u} = \vec{u}_h = \vec{g}$ and $\vmag \times
\vec{n} = \vmag_h \times \vec{n} = \vec{q}\times\vec{n}$ on $\partial \Omega$.
Then, although neither the true solution $U$ nor the numerical solution $U_h$
belong to $\EPspace$, the error $E$ defined in Definition \ref{defn:err}
satisfies homogeneous boundary conditions and hence belongs to $\EPspace$.
Thus, the error analysis in the previous section applies directly in this case.

On the other hand, if $U_h$ belongs to $\EPspaceh\setminus\EPspace$, then in
general $U_h$ does not satisfy the non-homogeneous boundary conditions exactly.
Hence we consider the splitting of the numerical solutions as,
\begin{equation}
\label{eq:num_soln_decomp}
U_h = U_h^0 + U^d,
\end{equation}
where $U_h^0 \in \EPspaceh$ solves,
\begin{equation}
\mathcal{N}_{EP}(U_h,V_h) = {\mathcal{N}}_{EP}(U_h^0 + U^d,V_h) = (F,V_h),\quad\forall\,V_h\in 	\EPspaceh,
\end{equation}
and $U^d$ is a known function that satisfies the non-homogeneous boundary
conditions accurately. That is, the unknown is now $U_h^0$ and the numerical
solution $U_h$ is formed through the sum in \eqref{eq:num_soln_decomp}. In this
article the function $U^d$ is approximated through a finite element space of
much higher dimension than $\EPspaceh$ to capture the boundary conditions
accurately and hence minimize discretization error. An alternate approach is to
represent $U^d$ in the same space as $U^0_h$ and then quantify the error due to
this approximation, for example  see \cite{chaudhry_2018}.

\subsection{Error estimate and contributions}
\label{sec:err_est_from_err_rep}
The error representation in Theorem~\ref{thm:err-rep-ep} requires the exact
solution $\Phi=(\vphi, \vbeta, \pi) \in \EPspace$ of
\eqref{eq:EP-dual-problem}.  Moreover, the adjoint form
\eqref{eq:EP-weak-formula} is linearized around  the true solution $U$ and the
approximate solution $U_h$. In practice, the adjoint solution itself must be
approximated in a finite element space $\mathcal{W}^h \subset \EPspace$ and is
linearized only around the numerical solution.  Let this approximation to the
adjoint be denoted by $\Phi_h=(\vphi_h, \vbeta_h, \pi_h) \in \mathcal{W}^h$.
This approximation leads to an error estimate from the error representation in
Theorem \ref{thm:err-rep-ep}. Let this error estimate be denoted by $\eta$.
That is, $\eta \approx (\Psi, E)$ such that,
\begin{equation}
   \label{eq:mhd_err_est}
		\eta = E_{mom} + E_{con} + E_{M},
	\end{equation}
	where,
	\begin{equation}
	\label{eq:err_contribs}
	\begin{aligned}
	E_{mom}&=(\vec{f},\vphi_h) - \bigg(\frac{1}{\Re}(\nabla\vec{u}_h,\nabla\vphi_h)+ ((\vec{u}_h\cdot\nabla)\vec{u}_h,\vphi_h)
	- (p_h, \diver\vphi_h)\\
	&+ \kappa((\curl\vmag_h)\times\vmag_h,\vphi_h)\bigg),\\
  E_{con}&=-(\diver\vec{u}_h,\pi_h),\\
  E_{M}&=-\frac{\kappa}{\Rm}(\curl\vmag_h,\curl\vbeta_h)+\kappa(\curl(\vec{u}_h\times\vmag_h),\vbeta_h)\\
	     &-\frac{\kappa}{\Rm}(\diver\vmag_h,\diver\vbeta_h).
	\end{aligned}
	\end{equation}
Here $E_{mom}$, $E_{con}$ and $E_{M}$ represent the momentum error
contribution, the continuity  error contribution and the magnetic error
contribution respectively.

To obtain an accurate error estimate we choose $\mathcal{W}^h$ to be of much
higher dimension than $\EPspaceh$ as is standard in adjoint based \textit{a
posteriori} error
estimation~\cite{Estep:Larson:00,Estep1995,Eriksson1995,CET+2016,CEG+2015,Estep:Larson:00,CBH+2014,CET-09,barth04}. 
Moreover, the inaccuracy caused by substituting the numerical solution in place
of true solution in the adjoint form is of higher order and shown to decrease in the
limit of refined discretization \cite{Estep:Larson:00,Cyr2014}.

\section{Numerical results}\label{sect:numerical-results}
In this section we present numerical results to verify the accuracy of the
error estimate \eqref{eq:mhd_err_est} and the and utility of the error
contributions in \eqref{eq:err_contribs}. 
The effectivity ratio, denoted Eff., characterizes how well the error estimate
approximates the true error, 
\begin{equation} 
\text{Eff.} =
\frac{\text{Error estimate}}{\text{True error}} = \frac{\eta}{(\Psi,E)}.\label{eq:eff-ratio}
\end{equation}
 The closer the effectivity is to 1, the better the
error estimate provided by our method.

We present two numerical examples here, the Hartmann problem in
\S\ref{sect:hartmann} which admits an analytic solution, and the magnetic lid
driven cavity \S\ref{sect:lid-driven}. Since there is no closed form solution
for the magnetic lid driven cavity, we use as reference a high order/fine mesh
solution to provide a high accuracy estimate for the true error. All the
following computations were carried out using  the finite element package
\texttt{Dolfin} in the \texttt{FEniCS} suite
\cite{AlnaesBlechta2015a,LoggMardalEtAl2012a,LoggWellsEtAl2012a}.

For all experiments, we chose different polynomial orders of Lagrange spaces
for the product space $\EPspaceh$ and choose the adjoint space $\mathcal{W}^h$
such that it is one higher polynomial degree in each variable.
 The computational domain for all problems is chosen to be a unit length
 square, $\Omega := [-\tfrac{1}{2}, \frac{1}{2}]^2\subset\R^2$. The mesh is
 a simplicial uniform mesh with the total number of elements denoted by $\#
 Elements$.

\subsection{Hartmann flow in two dimensions}\label{sect:hartmann}
Our first results concern the so-called Hartmann problem
\cite{muller_ulrich_buhler_leo_2010}. This problem models the one-dimensional
flow of a conducting fluid in a channel and forms both a momentum boundary
layer (viscous boundary layer), and a layer formed by the diffusion of the
magnetic field that influences the flow due to the Lorentz force (a Hartmann
layer).  In this case we take consider a square channel as the computational
domain, however the analytic solution is only a one-dimensional profile, as
described in the beginning of the section. This problem admits an analytic
solution \cite{phillips_elman_cyr_shadid_pawlowski_2014},
$\vec{u}=\begin{bmatrix}u_x, 0\end{bmatrix}^T$, $\vmag=\begin{bmatrix}b_x,
1\end{bmatrix}^T,\, p$ where
\begin{subequations}
\begin{align}
u_x(y) &=\frac{G\,\Re(\cosh(\Ha/2) - \cosh(\Ha y))}{2\Ha\sinh(\Ha/2)},\\
B_x(y) &=\frac{G(\sinh(\Ha y) - 2\sinh(\Ha/2)y)}{2\kappa\sinh(\Ha/2)},\\
p(x) &= -Gx -\kappa B_x^2/2,
\end{align}\label{eq:hartmann}
\end{subequations}
and $G=-\frac{dp}{dx}$ is an arbitrary pressure drop that we choose to
normalize the maximum velocity $|u_x(y)|$ to 1.

\subsubsection{Problem parameters and QoI}
The values of the nondimensionalized constants are chosen as follows:
$\Re=16,\Rm=16, \kappa=1$ which produce a Hartmann number of $\Ha=16$. The QoI
is chosen as the average velocity across the flow over a slice. To this end,
define
\begin{equation}
\Omega_c:=\left[-\tfrac{1}{4},\tfrac{1}{2}\right]\times\left[-\tfrac{1}{4},\tfrac{1}{4}\right]
\end{equation}
and consequently $\mathbbm{1}_{\Omega_c}$ the characteristic function on
$\Omega_c$. We choose $\Psi$ to be
$\Psi=\begin{bmatrix}\mathbbm{1}_{\Omega_c}, 0, 0, 0, 0\end{bmatrix}^T$
so that the QoI \eqref{eq:QoI} thus reduces to
\begin{equation}
(\Psi, U) = (\mathbbm{1}_{\Omega_c}, u_x).
\end{equation}
This has a physical interpretation of the capturing the flow rate across this
slice of the channel, $\Omega_c$.

\subsubsection{Numerical results and discussion} 
The error contributions of \eqref{eq:mhd_err_est} as well as effectivity ratios using
different order polynomial spaces are presented in
Table \ref{tab:Hartmann-p2p1p1}, Table \ref{tab:Hartmann-p2p2p1-linerr},
Table \ref{tab:Hartmann-p3p2p2-linerr}, and
Table \ref{tab:Hartmann-p3p2p2-nolinerr}. The effectivity ratio in tables
Table \ref{tab:Hartmann-p2p1p1} and Table \ref{tab:Hartmann-p2p2p1-linerr} is quite
close to 1 indicating the accuracy of the error estimate. The error estimate in
Table \ref{tab:Hartmann-p3p2p2-linerr} is not as accurate due to linearization
error incurred by replacing the true solution by the approximate solution in
the definition of the adjoint as discussed in \S \ref{sec:err_est_from_err_rep}.
This may be verified by linearizing the adjoint weak form around both the true
(which we know for this example) and the approximate solutions.  These results
are shown in Table \ref{tab:Hartmann-p3p2p2-nolinerr} and now the error estimate
is again accurate. 

In Table \ref{tab:Hartmann-p2p1p1} we use the lowest order tuple of Lagrange
spaces, $(\P^2, \P^1, \P^1)$ for the variables $(\vec{u},\vmag,p)$. In this
case, the error is largely dominated by the contributions $E_{con}$ and $E_M$.
We greatly reduce the error in $E_M$ by using a higher degree Lagrange space,
$\P^2$, for $\vmag$ as demonstrated in table
Table \ref{tab:Hartmann-p2p2p1-linerr}. However, this does not reduce the
magnitude of the total error much (about $5\%$) which is still dominated by the
contribution $E_{con}$. The contribution $E_{con}$ is not significantly
affected by the finite dimensional space for $\vmag$. Now finally, in
Table \ref{tab:Hartmann-p3p2p2-linerr} we use a higher order tuple $(\P^3, \P^2,
\P^2)$ for $(\vec{u},\vmag, p)$ and the total error drops by two orders of
magnitude. 

\begin{table}[!ht]
	\centering
	\begin{tabular}{|c|c|c|c|c|c|c||}
	\hline
	\# Elements&  True Error & Eff. & $E_{mom}$ & $E_{con}$ & $E_M$ \\
	\hline
	1600 & 2.76e-04 & 1.00 & 4.53e-06 & -2.28e-04 &5.00e-04\\
  \hline
  6400 & 6.98e-05 & 1.00 & 1.29e-06 & -6.23e-05 &1.31e-04\\
  \hline
  14400 & 3.11e-05 & 1.00 & 6.05e-07 & -2.86e-05 &5.91e-05\\
  \hline
  25600 & 1.75e-05 & 1.00 & 3.49e-07 & -1.63e-05 &3.35e-05\\
  \hline
	\end{tabular}
\caption{Error in $(u_x, \mathbbm{1}_{\Omega_c})$ for the Hartmann problem \S\ref{sect:hartmann}, with $\mathbbm{1}_{\Omega_c} = [-\tfrac{1}{4}, \tfrac{1}{2}]\times [-\tfrac{1}{4},\tfrac{1}{4}]$. The finite dimensional space here is $(\P^2, \P^1,\P^1)$  for $(\vec{u},\vmag, p)$.}
\label{tab:Hartmann-p2p1p1}
\end{table}
\begin{table}[!ht]
	\centering
	\begin{tabular}{|c|c|c|c|c|c|c|c||}
    \hline
    \# Elements& True Error & Eff. & $E_{mom}$ & $E_{con}$ & $E_M$ \\
    \hline
    1600 & -2.25e-04 & 1.02 & 1.08e-06 & -2.27e-04 &-4.79e-06\\
    \hline
    6400 & -6.13e-05 & 1.04 & 1.04e-06 & -6.23e-05 &-2.18e-06\\
    \hline
    14400 & -2.81e-05 & 1.04 & 5.98e-07 & -2.86e-05 &-1.13e-06\\
    \hline
    25600 & -1.60e-05 & 1.04 & 3.76e-07 & -1.64e-05 &-6.81e-07\\
    \hline
	\end{tabular}
\caption{Error in $(u_x, \mathbbm{1}_{\Omega_c})$ for the Hartmann problem \S\ref{sect:hartmann}. The finite dimensional space here is $(\P^2, \P^2,\P^1)$ for $(\vec{u},\vmag, p)$.}
\label{tab:Hartmann-p2p2p1-linerr}
\end{table}
\begin{table}[!ht]
	\centering
	\begin{tabular}{|c|c|c|c|c|c|c|c||}
	\hline
	\# Elements& True Error & Eff. & $E_{mom}$ & $E_{con}$ & $E_M$ \\
	\hline
	1600 & 1.23e-06 & 1.21 & 3.97e-07 & -4.15e-06 &5.24e-06\\
  \hline
  6400 & 1.46e-07 & 1.47 & 9.23e-08 & -5.07e-07 &6.29e-07\\
  \hline
  14400 & 4.97e-08 & 1.63 & 3.84e-08 & -1.40e-07 &1.83e-07\\
  \hline
  25600 & 2.47e-08 & 1.73 & 2.07e-08 & -5.44e-08 &7.64e-08\\
  \hline
	\end{tabular}
\caption{Error in $(u_x, \mathbbm{1}_{\Omega_c})$ for the Hartmann problem
\S\ref{sect:hartmann}. The finite dimensional space here is $(\P^3,
\P^2,\P^2)$  for $(\vec{u},\vmag, p)$.  Here, we approximate the true
solution with the computed solution which results in linearization error. For
this accurate a solution, this deteriorates the quality of the estimate which
in turn results in a efficiency further from 1.  This is confirmed in
Table \ref{tab:Hartmann-p3p2p2-nolinerr} where we use the true solution and the
effectivity is again close to 1.}
\label{tab:Hartmann-p3p2p2-linerr}
\end{table}
\begin{table}[!ht]
	\centering
	\begin{tabular}{|c|c|c|c|c|c|c|c||}
	\hline
	2d Elem.& True Error & Eff. & $E_{mom}$ & $E_{con}$ & $E_M$ \\
	\hline
  1600 & 1.23e-06 & 1.00 & 2.75e-07 & -4.39e-06 &5.34e-06\\
  \hline
  6400 & 1.46e-07 & 1.00 & 5.97e-08 & -5.60e-07 &6.46e-07\\
  \hline
  14400 & 4.97e-08 & 1.00 & 2.35e-08 & -1.63e-07 &1.89e-07\\
  \hline
  25600 & 2.47e-08 & 1.00 & 1.22e-08 & -6.65e-08 &7.90e-08\\
  \hline
	\end{tabular}
\caption{Error in $(u_x, \mathbbm{1}_{\Omega_c})$ for the Hartmann problem, \S\ref{sect:hartmann}. The finite dimensional space here is $(\P^3, \P^2,\P^2)$  for $(\vec{u},\vmag, p)$.
 No linearization error is present here because we use the true solution in the definition of the adjoint. }
\label{tab:Hartmann-p3p2p2-nolinerr}
\end{table}

\subsection{Magnetic Lid Driven Cavity}\label{sect:lid-driven}
\subsubsection{Regularization and solution method}
The magnetic lid driven cavity is another common benchmark problem for verifying
MHD codes \cite{phillips_elman_cyr_shadid_pawlowski_2014, Sivasankaran2011}.
However, the standard lid velocity is discontinuous and therefore obtains at most
$H^{1/2-\epsilon}$ regularity in two dimensions with $\epsilon>0$. By the converse of the
trace theorem and the Sobolev inequality \cite{ern_guermond_2011,
brenner_scott_2011}, the solution $u_x$ cannot obtain $H^1$ regularity on the
interior. Indeed, in this situation, we do not even have well-posedness of the
primal problem, so there is not real hope for error analysis. This issue has
been address in a purely fluid context \cite{Hamouda2017,
lee_dowell_balajewicz_2018}. In both cases, a regularization of the lid
velocity is proposed to mitigate theoretical issues (in the former) and the
ability to achieve higher Reynold's numbers (in the latter). In this work, we
use a similar regularization to the one proposed in
\cite{lee_dowell_balajewicz_2018}, a polynomial regularization of the lid
velocity,
\[u_{top}(x)=C\left(x-\tfrac{1}{2}\right)^2\left(x+\tfrac{1}{2}\right)^2,\]
with $C$ chosen such that 
\[\int_{-1/2}^{1/2}u_{top}(x)\d x = 1.\] 
The boundary conditions are imposed as $\vec{g}(x,
0.5)=\begin{bmatrix}u_{top},0\end{bmatrix}^T$ on the top face and zero on the
rest of the boundary. The boundary conditions for the magnetic field are
$\vec{q}=\begin{bmatrix}-1,0\end{bmatrix}^T$ so that
$\vec{b}\times\vec{n}=\begin{bmatrix}-1,0\end{bmatrix}^T\times\vec{n}$ on
$\p\Omega$.  To get a qualitative measure of the validity of the regularized
problem, we show plot of the velocity profile for a fixed Reynold's number
$\Re=5000$ and varying magnetic Reynold's numbers $\Rm$ in Figure
\ref{fig:liddriven}.  These plots are qualitatively similar to Figure 1 in
\cite{phillips_elman_cyr_shadid_pawlowski_2014} (for which an un-regularized
lid velocity is used), which gives a good indication that the regularized
version produces qualitatively similar features.
\begin{figure}[!ht]
    \centering
		\subfloat[$\Rm=0.1$]{\includegraphics[height=0.285\textwidth]{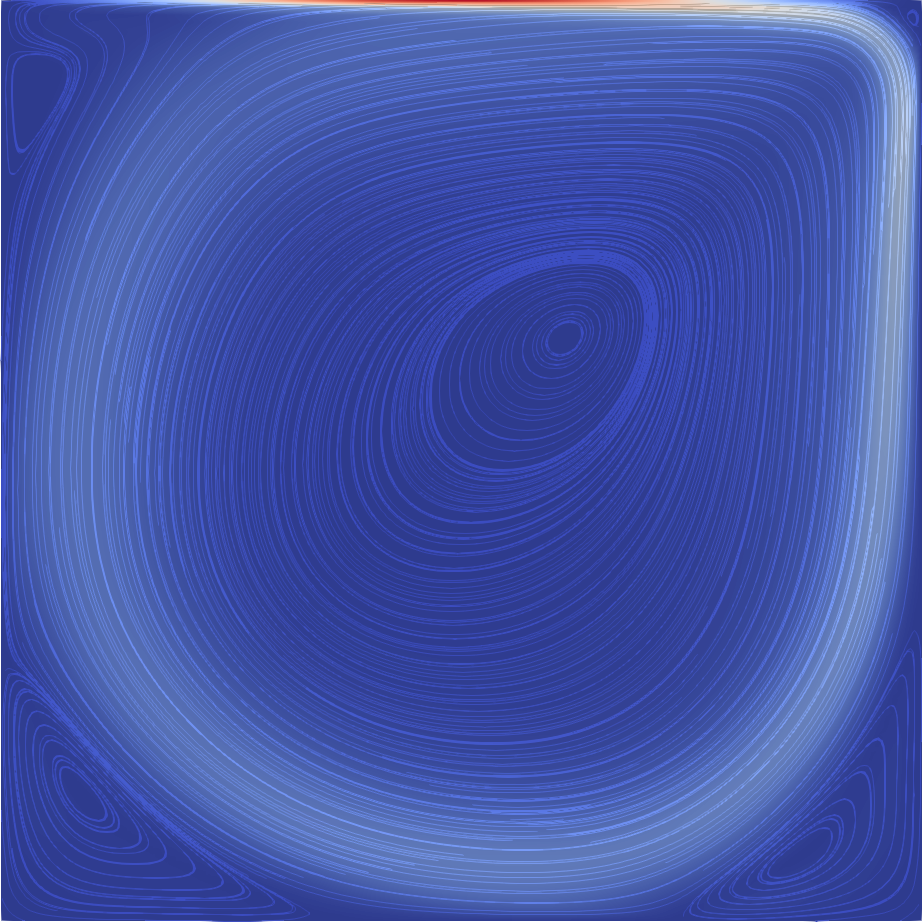}}
		\phantom{-}
		\subfloat[$\Rm=0.5$]{\includegraphics[height=0.285\textwidth]{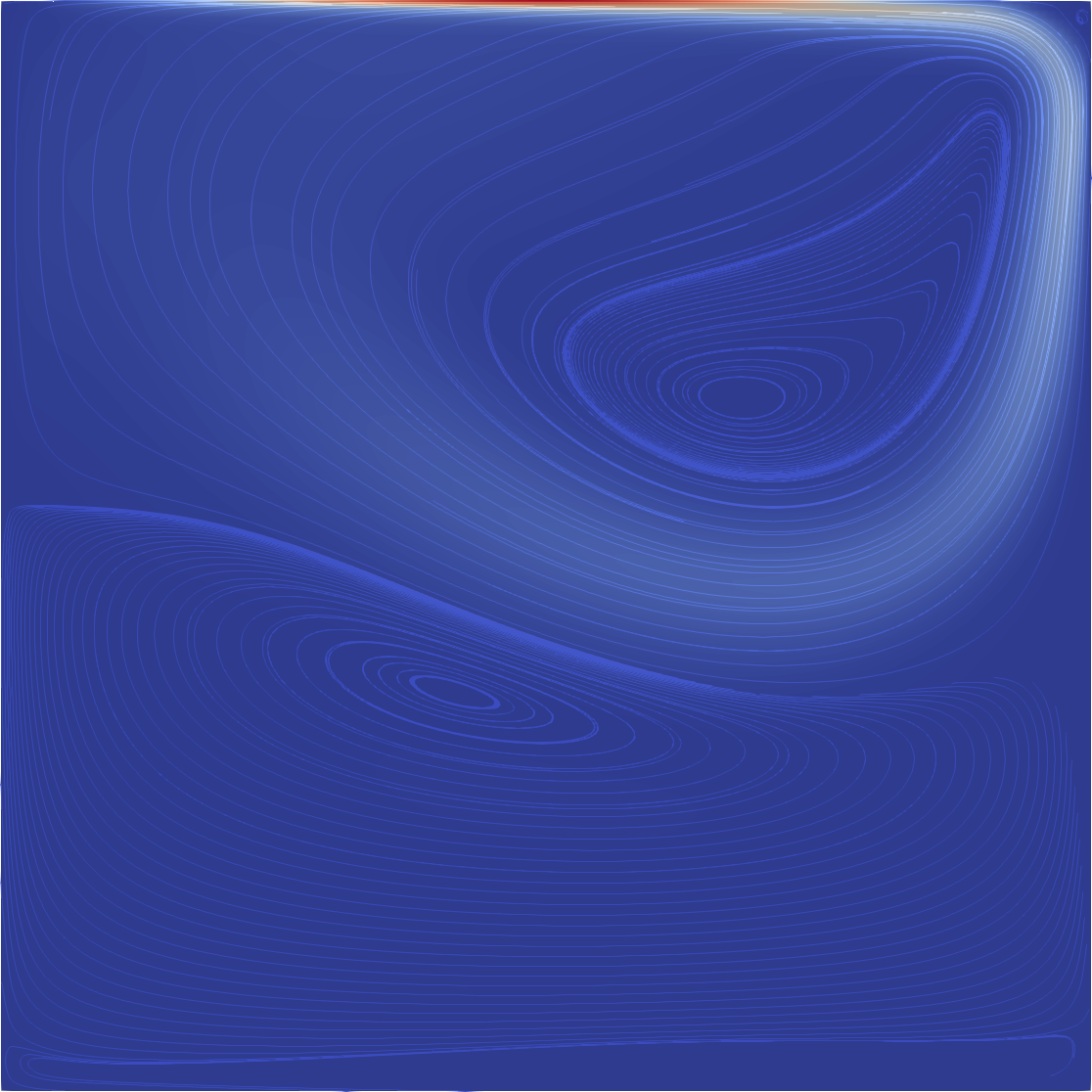}}
		\phantom{-}
		\subfloat[$\Rm=5.0$]{\includegraphics[height=0.285\textwidth]{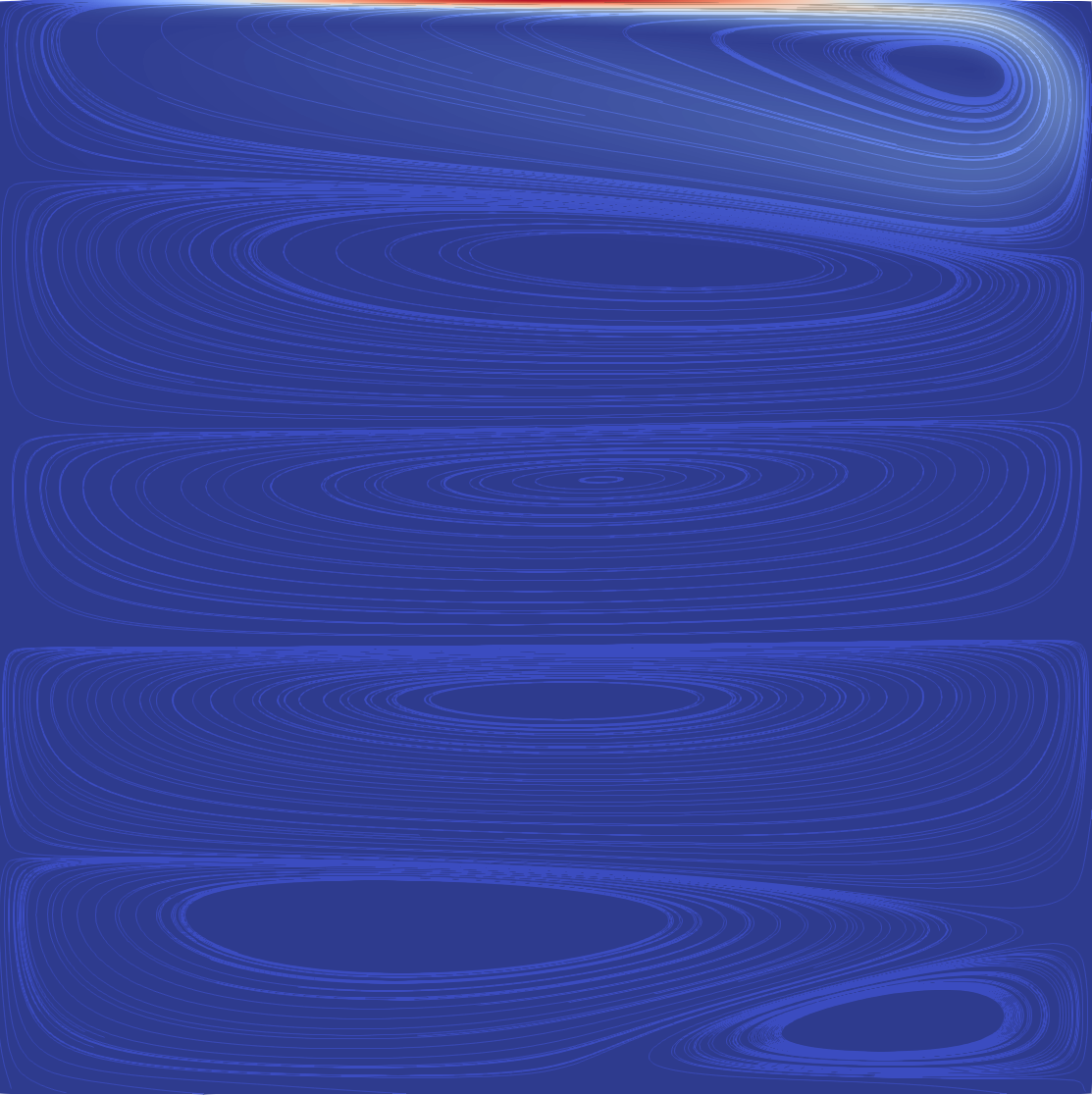}}
		\subfloat[]{\includegraphics[height=0.285\textwidth]{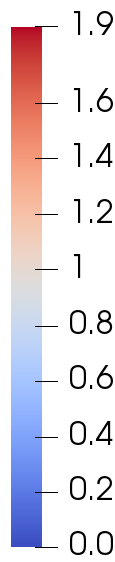}}		
    \caption{Plots of the $\|\vec{u}\|_{\R^d}$ for the lid driven cavity
    \S\ref{sect:lid-driven} with added streamlines.  We use a normalization on
    the lid velocity over a variety of magnetic Reynold's numbers, $\Rm$.  The
    other nondimensionalized parameters $\Re=5000, \kappa=1$ for all of these
    plots.}
	 \label{fig:liddriven}
\end{figure}

Furthermore, since Newton's method requires a good initial guess for this
problem, we use a homotopic sequence of initial guesses to achieve convergence
to high $\Re$. Specifically we run the problem for a moderate value of
$\Re=200$ for example, and then use the solution produced by the solver as the
initial guess for a larger value e.g. $\Re=1000$ until we have achieved the
desired value.  Figure \ref{fig:homotopy} shows the intermediate values in this
sequence to solve a problem with $\Re=1000$.
\begin{figure}[!ht]
    \centering
		\subfloat{\includegraphics[height=0.27\textwidth]{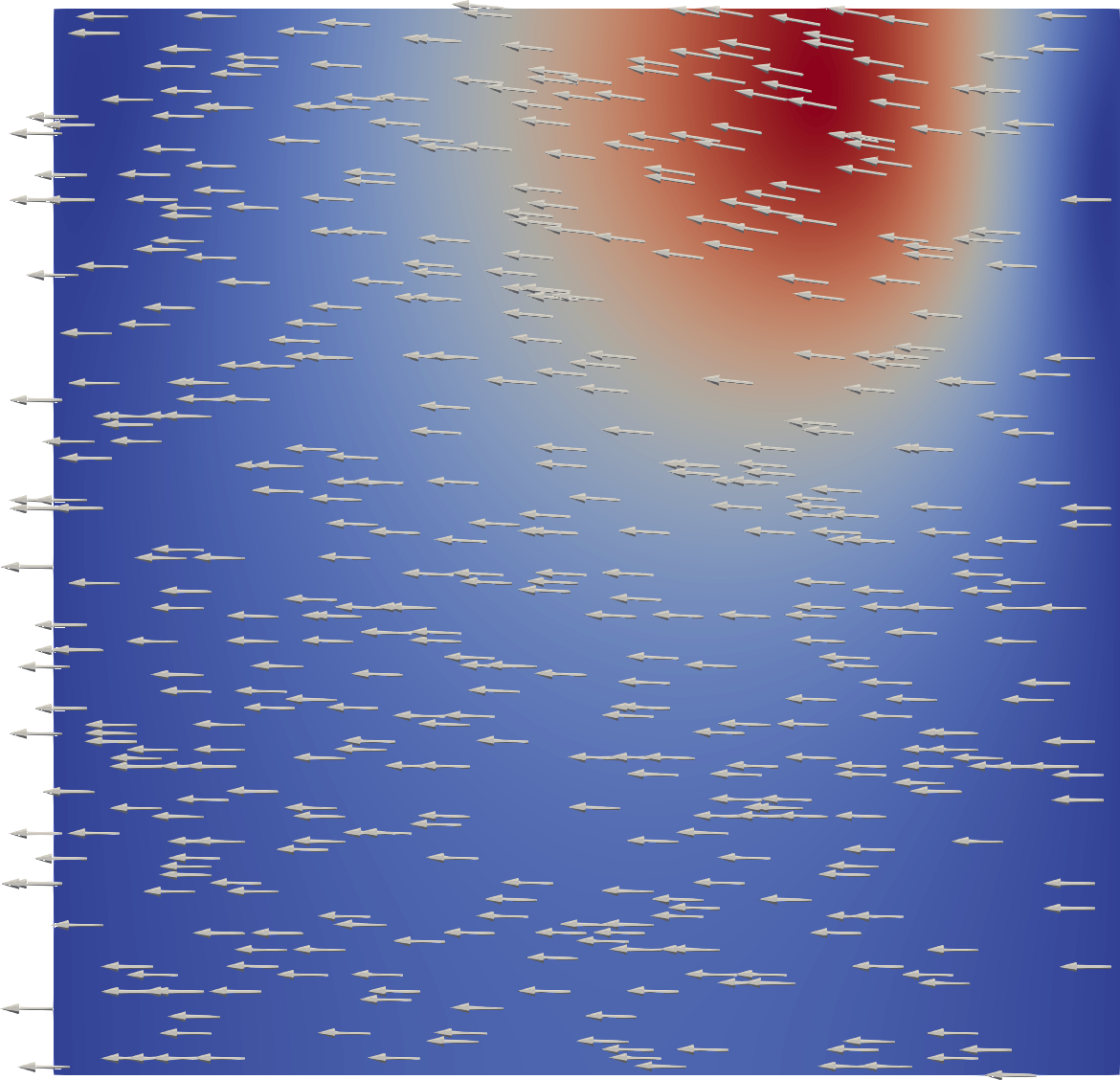}}
		\phantom{'}
		\subfloat{\includegraphics[height=0.27\textwidth]{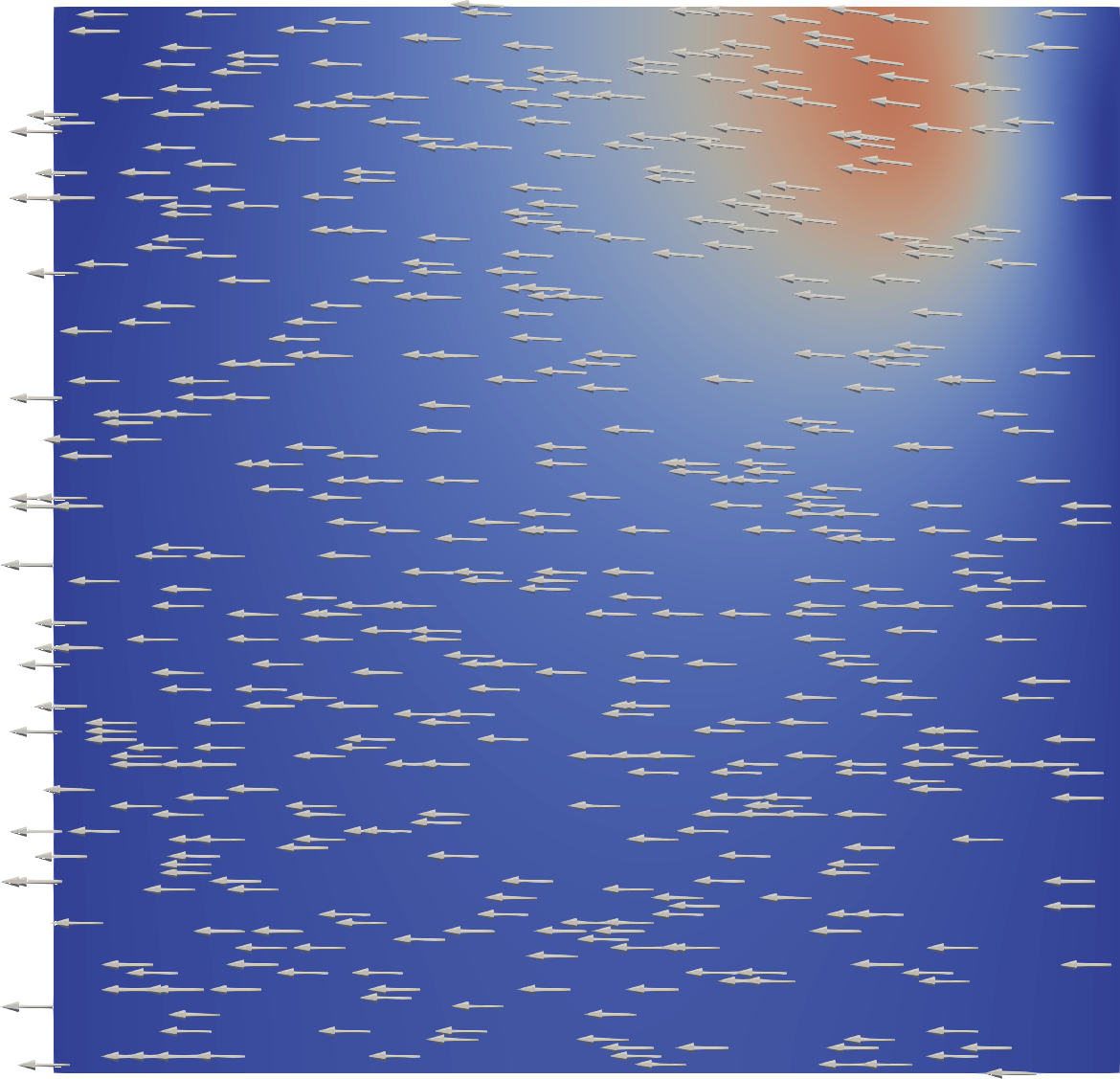}}
		\phantom{'}
		\subfloat{\includegraphics[height=0.27\textwidth]{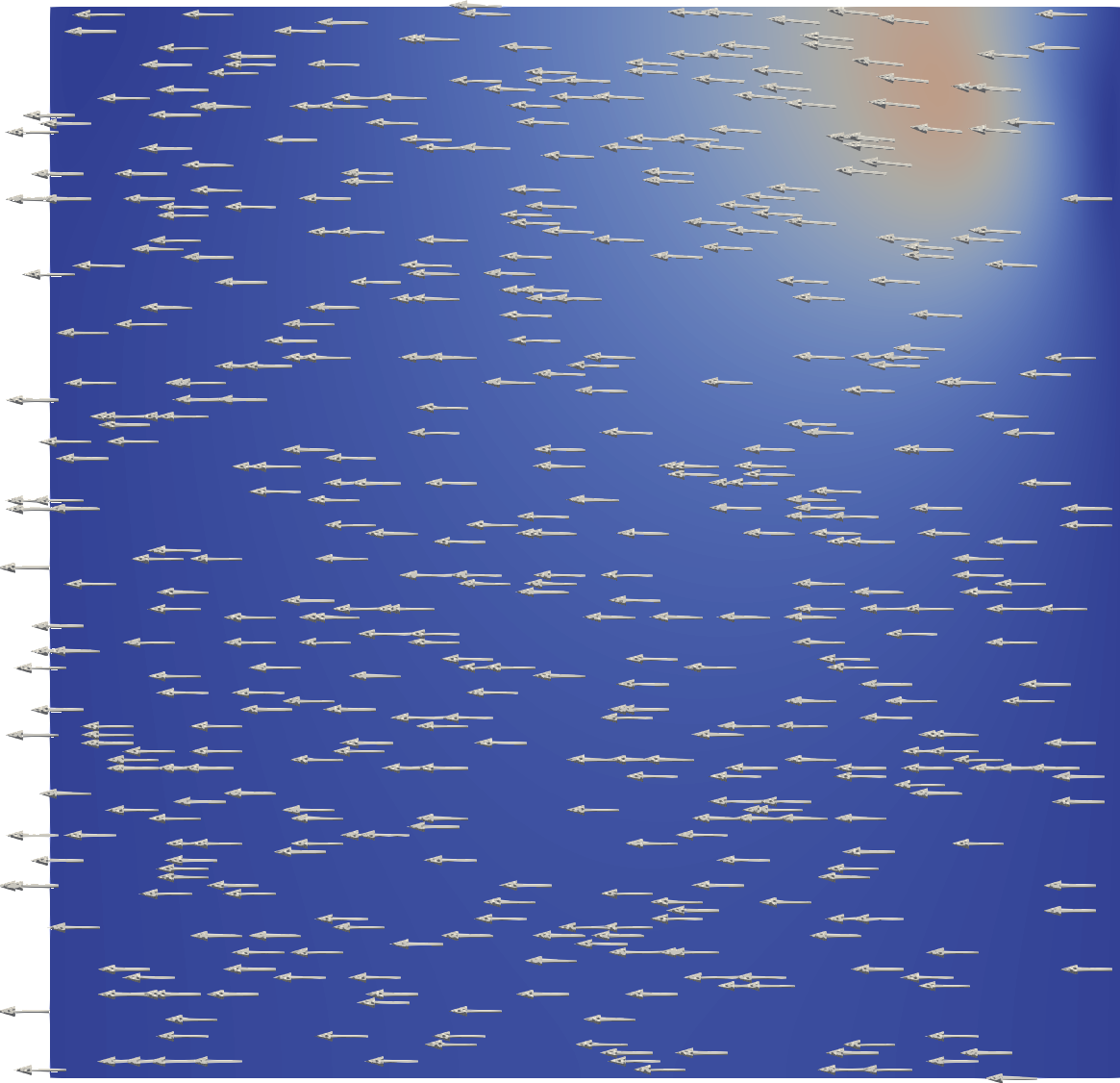}}
		\subfloat[]{\includegraphics[height=0.27\textwidth]{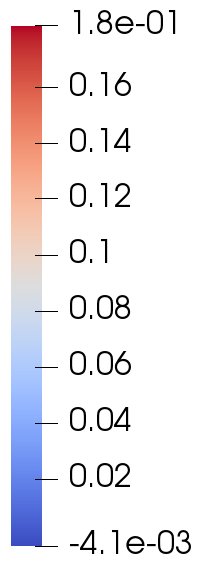}}	\\
		\subfloat[$\Re=200$]{\includegraphics[height=0.27\textwidth]{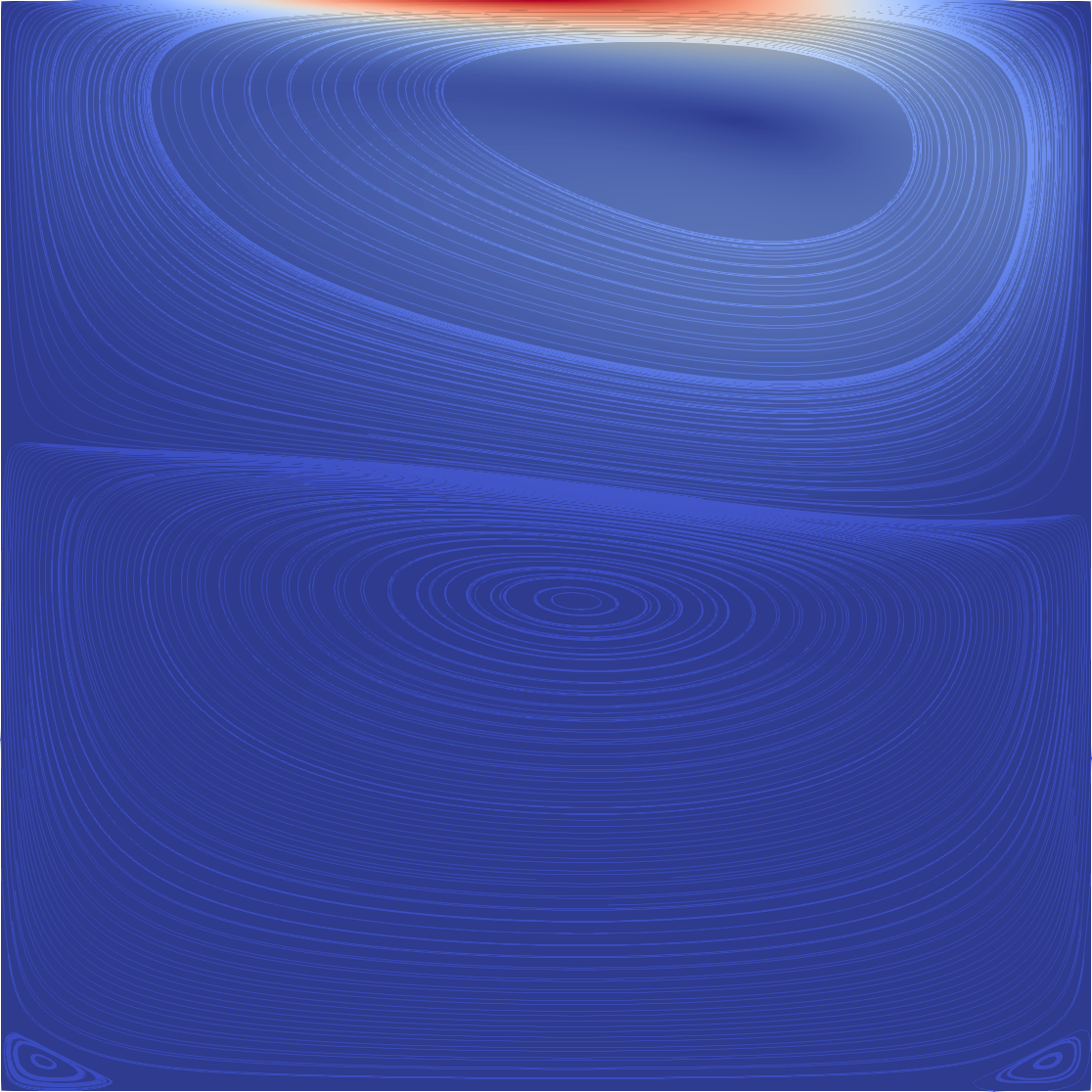}}
		\phantom{-}
		\subfloat[$\Re=500$]{\includegraphics[height=0.27\textwidth]{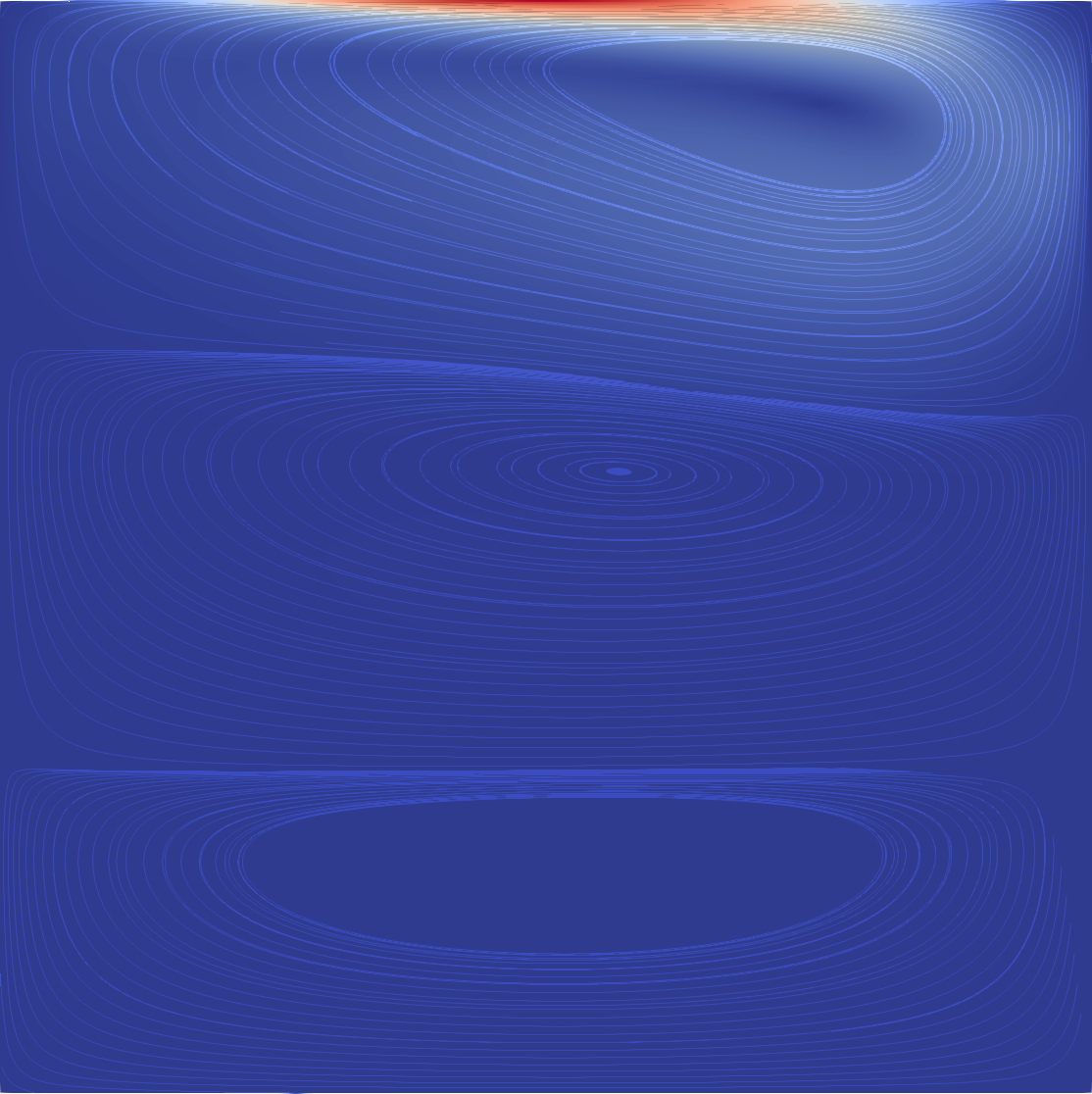}}
		\phantom{-}
		\subfloat[$\Re=1000$]{\includegraphics[height=0.27\textwidth]{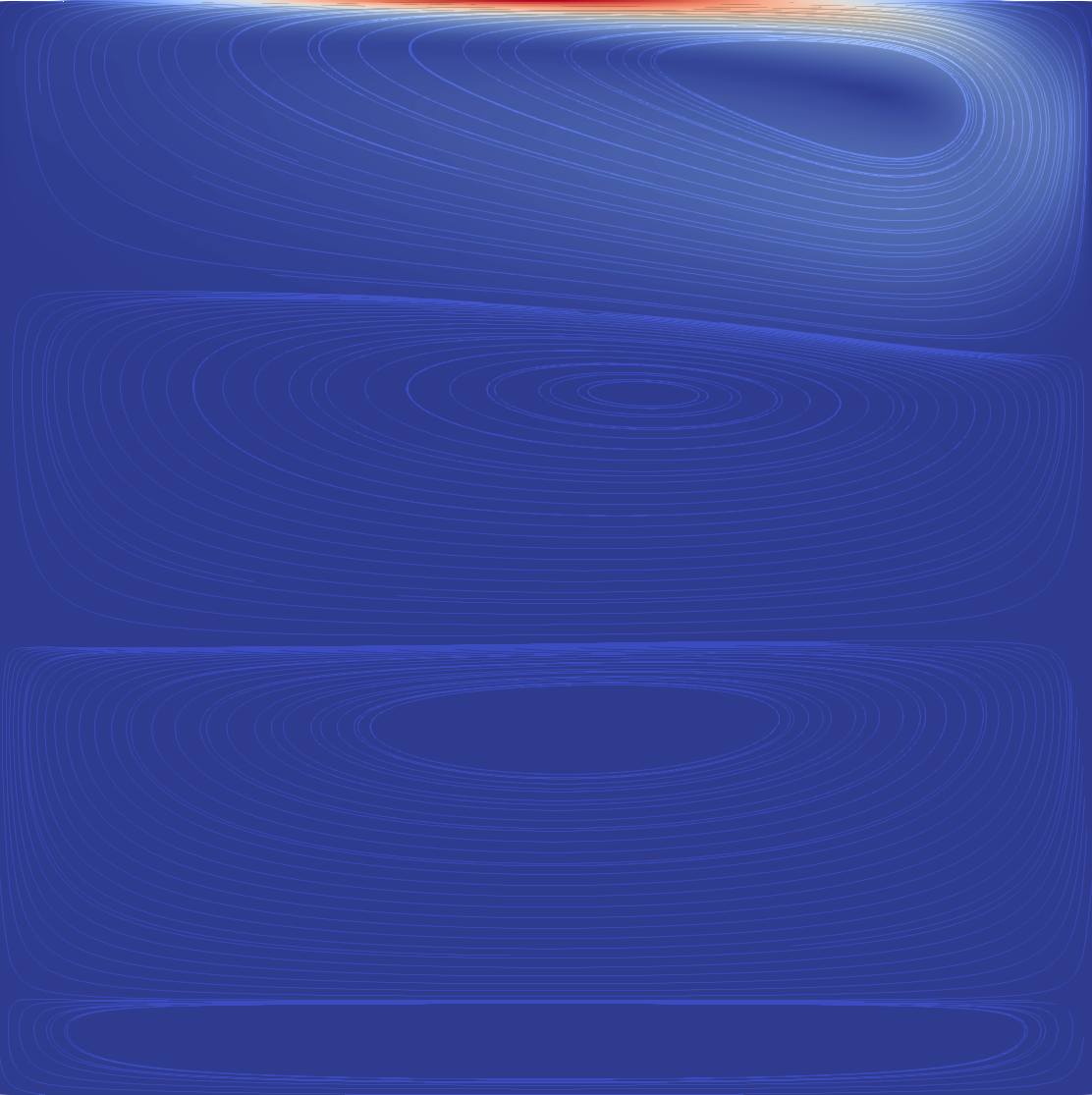}}
		\subfloat[]{\includegraphics[height=0.27\textwidth]{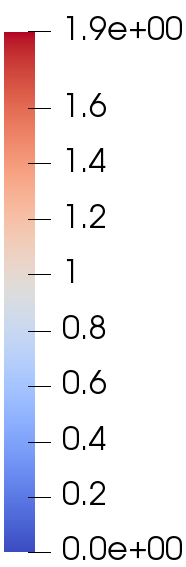}}\\
		\caption{Demonstrating the homotopy parameter strategy to achieve high
		fluid Reynold's numbers as described in \S\ref{sect:lid-driven}. The
		other nondimensionalized parameters $\Rm=5.0, \kappa=1$ for all of these
		plots. The top row is colored according the $b_y$ and with the arrows
		representing the vector $\vmag$. The bottom row is colored according to
    $\|\vec{u}\|_{\R^d}$, with added streamlines.}
	 \label{fig:homotopy}
\end{figure}

\subsubsection{Problem parameters and results}
We consider our QoI \eqref{eq:QoI} with $\Psi = \begin{bmatrix}0, 0, 0, \mathbbm{1}_{\Omega_c},
0\end{bmatrix}^T$ where now
\begin{equation}
\Omega_c:=\left[-\tfrac{1}{4},\tfrac{1}{4}\right]\times\left[0,\tfrac{1}{2}\right],
\end{equation}
so that the QoI $(\Psi, U) = (\mathbbm{1}_{\Omega_c}, b_y)$ gives a measure of the
induced magnetic field in the upper middle half of the box. See Figure \ref{fig:homotopy}
for plots of the induced field $b_y$ as a function of Reynold's number
$\Re$.

Since there is no analytic solution for this problem, we compute solution on
a $400\times 400$ mesh in the space $(\P^3, \P^2, \P^2)$ for $(\vec{u}, \vmag,
p)$.  We consider the QoI obtained from this very high resolution reference
solution as a the true solution to  compute the error in the denominator of the
effectivity ratio \eqref{eq:eff-ratio}.  The effectivity ratio and error
contributions for $\Re=1000$ and $\Re=2000$ are shown in
Tables~\ref{tab:lid_driven-p2p1p1-Re-1000},
\ref{tab:lid_driven-p3p1p2-Re-1000}, \ref{tab:lid_driven-p2p1p1-Re-2000} and
\ref{tab:lid_driven-p3p1p2-Re-2000}.  The error estimate $\eta$ is deemed
accurate since all effectivity ratios are close to 1.

We first study the lowest order case, namely using the space $(\P^2, \P^1,
\P^1)$ for $(\vec{u}, \vmag, p)$ in Table \ref{tab:lid_driven-p2p1p1-Re-1000}
where $\Re=1000$ and Table \ref{tab:lid_driven-p2p1p1-Re-2000} where $\Re=2000$.
For both $\Re=2000$ and $\Re=1000$, the error contributions are not drastically
different in magnitude, and become even more similar as the mesh is refined. We
also note that all contributions, and in particular the true error, are larger
in magnitude for the case $\Re=2000$. 

For the next experiment, we consider a higher order space for the velocity pair
$(\vec{u}, p)$ namely $(\P^3, \P^1, \P^2)$ for $(\vec{u},\vmag,p)$ in Table
\ref{tab:lid_driven-p3p1p2-Re-1000} for $\Re=1000$ and Table
\ref{tab:lid_driven-p3p1p2-Re-2000} for $\Re=2000$. In both cases, the error is
now dominated by the contribution $E_M$. The case of $\Re=2000$ is particularly
interesting, as the error increases as the mesh is refined from
1600 elements to 3600 elements. This seemingly anomalous behavior is explain by
examining the error contributions. For $\# Elements = 1600$ we have that
$E_{mom}+E_{con}$ has magnitude comparable to that of $E_M$ but opposite
sign, and hence there is cancellation of error.  For $\# Elements = 3600$,
the magnitude of $E_{mom}+E_{con}$ is much less than that of $E_M$ and
hence the total error increases as there is less cancellation of error.
Hence, adjoint based analysis not only quantifies the error, it also helps
in diagnosing such anomalous behavior.
\begin{table}[!ht]
	\centering
	\begin{tabular}{|c|c|c|c|c|c|c|c||}
	\hline
	\# Elements& True Error & Eff. & $E_{mom}$ & $E_{con}$ & $E_M$ \\
		\hline
		1600 & -3.93e-05 & 0.99 & -1.05e-05 & -2.47e-05 &-3.78e-06\\
		\hline
		3600 & -9.50e-06 & 0.97 & -2.23e-06 & -5.23e-06 &-1.74e-06\\
		\hline
		6400 & -3.41e-06 & 0.98 & -8.12e-07 & -1.52e-06 &-9.87e-07\\
		\hline
		10000 & -1.61e-06 & 0.98 & -3.64e-07 & -5.81e-07 &-6.33e-07\\
		\hline
 	\end{tabular}
\caption{Error estimates for $(b_y, \mathbbm{1}_{\Omega_c})$ for the lid driven cavity
\S\ref{sect:lid-driven}. The finite dimensional space here is
$(\P^2, \P^1,\P^1)$ for $(\vec{u},\vmag, p)$. We
use a very high resolution reference solution on a 400x400=160000 element mesh and $(\P^3, \P^2,
\P^2)$ elements. The parameters are $\Re=1000,\Rm=0.4,\kappa=1$.}
\label{tab:lid_driven-p2p1p1-Re-1000}
\end{table}

\begin{table}[!ht]
	\centering
	\begin{tabular}{|c|c|c|c|c|c|c||}
	\hline
	\# Elements&  True Error & Eff. & $E_{mom}$ & $E_{con}$ & $E_M$ \\
	\hline
	1600 & -5.37e-06 & 0.98 & -4.65e-07 & -9.75e-07 &-3.81e-06\\
	\hline
	3600 & -1.95e-06 & 0.99 & -5.49e-08 & -1.27e-07 &-1.75e-06\\
	\hline
	6400 & -1.03e-06 & 1.00 & -1.06e-08 & -2.76e-08 &-9.87e-07\\
	\hline
	10000 & -6.45e-07 & 1.00 & -2.89e-09 & -8.04e-09 &-6.33e-07\\
	\hline
 	\end{tabular}
\caption{Error estimates for $(b_y, \mathbbm{1}_{\Omega_c})$ for the lid driven cavity
\S\ref{sect:lid-driven}. The finite dimensional space here is
$(\P^2, \P^2,\P^1)$ for $(\vec{u},\vmag, p)$. We
use a very high resolution reference solution on a 400x400=160000 element mesh and $(\P^3, \P^2,
\P^2)$ elements. The parameters are $\Re=1000,\Rm=0.4,\kappa=1$.}
\label{tab:lid_driven-p3p1p2-Re-1000}
\end{table}

\begin{table}[!ht]
	\centering
	\begin{tabular}{|c|c|c|c|c|c|c|c||}
	\hline
	\# Elements& True Error & Eff. & $E_{mom}$ & $E_{con}$ & $E_M$ \\
	\hline
	1600 & -8.01e-05 & 1.10 & -3.65e-05 & -5.70e-05 &5.63e-06\\
	\hline
	3600 & -2.04e-05 & 0.98 & -5.69e-06 & -1.66e-05 &2.25e-06\\
	\hline
	6400 & -5.92e-06 & 0.96 & -1.84e-06 & -5.06e-06 &1.19e-06\\
	\hline
	10000 & -2.07e-06 & 0.96 & -8.17e-07 & -1.91e-06 &7.41e-07\\
	\hline
 	\end{tabular}
\caption{Error estimates for $(b_y, \mathbbm{1}_{\Omega_c})$ for the lid driven cavity
\S\ref{sect:lid-driven}. The finite dimensional space here is
$(\P^2, \P^1,\P^1)$ for $(\vec{u},\vmag, p)$. We
use a very high resolution reference solution on a 400x400=160000 element mesh and $(\P^3, \P^2,
\P^2)$ elements. The parameters are $\Re=2000,\Rm=0.4,\kappa=1$.}
\label{tab:lid_driven-p2p1p1-Re-2000}
\end{table}

\begin{table}[!ht]
	\centering
	\begin{tabular}{|c|c|c|c|c|c|c|c||}
	\hline
	\# Elements& True Error & Eff. & $E_{mom}$ & $E_{con}$ & $E_M$ \\
	\hline
	1600 & 1.31e-06 & 0.78 & -1.58e-06 & -3.47e-06 &6.08e-06\\
	\hline
	3600 & 1.51e-06 & 0.96 & -1.91e-07 & -5.29e-07 &2.17e-06\\
	\hline
	6400 & 1.02e-06 & 0.98 & -3.87e-08 & -1.28e-07 &1.17e-06\\
	\hline
	10000 & 6.94e-07 & 0.99 & -1.07e-08 & -4.04e-08 &7.38e-07\\
	\hline
	\end{tabular}
\caption{Error estimates for $(b_y, \mathbbm{1}_{\Omega_c})$ for the lid driven cavity
\S\ref{sect:lid-driven}. The finite dimensional space here is
$(\P^3, \P^2,\P^1)$ for $(\vec{u},\vmag, p)$. We
use an very high resolution reference solution on a 400x400=160000 element mesh and $(\P^3, \P^2,
\P^2)$ elements. The parameters are $\Re=2000,\Rm=0.4,\kappa=1$.}
\label{tab:lid_driven-p3p1p2-Re-2000}
\end{table}

\subsection{Illustrative compute time comparison of the primal and adjoint problems}
  In this section we study CPU times for the Hartmann problem of
  \S\ref{sect:hartmann} using $(\P^2,\P^1,\P^1)$ for $(\vec{u},\vmag, p)$. In
  particular this corresponds to the experiment in Table
  \ref{tab:Hartmann-p2p1p1}. We compare the CPU time of numerically
  solving the adjoint problem with with solving the discrete forward problem
  \eqref{eq:primal-FE}. The adjoint problem is solved in a higher order space
  $(\P^3,\P^2,\P^2)$, but since it is linear, it is not obvious how it compares
  in terms of computational cost to the primal problem.  The CPU times
  are shown in Table \ref{tab:timings} \footnote{These experiments were carried
    out using a dual-socket workstation with two Intel Xeon E5-2687W v2 for
  a total of 16 physical cores and 32 threads.}. The CPU time required for the
  adjoint problem is less in all cases than the CPU time required for solving
  the primal problem. We note that these results depend on the choice of linear
  and nonlinear solvers and preconditioners; here we are simply using Newton's
  method and direct linear solvers for  the primal problems and direct linear
  solvers for the adjoint problems.
 \begin{table}
  \centering
  \begin{tabular}{| c | c | c |}
    \hline
    \# Elements & Primal solve time (s) & Adjoint solve time (s)\\
    \hline
    1600 & 0.73 & 0.45\\
    \hline
    6400 & 3.40 & 1.62\\
    \hline
    14400 & 6.28 & 4.09\\
    \hline
    25600 & 11.70 & 8.01\\
    \hline
  \end{tabular}
  \caption{CPU times for the primal problem (using $(\P^2,\P^1,\P^1)$) and adjoint problem (using $(\P^3,\P^2, \P^2)$) corresponding to the results in Table \ref{tab:Hartmann-p2p1p1}. }
  \label{tab:timings}
\end{table}

\section{Derivation of the weak adjoint and well-posedness}\label{sect:theoretical-results}
In this section we provide the details of computing the adjoint to exact
penalty weak form following the theory in \S\ref{sect:error-analysis-abs}.
Then we use a standard saddle point argument to demonstrate the well-posedness
of this new adjoint problem \eqref{eq:EP-dual-problem}. We take inspiration for
these proofs  from \cite{gunzburger_meir_peterson_1991}. To simplify notation
in this section, we define 
\begin{equation}
\begin{aligned}
&\vec{s}:=\vec{u}+\vec{u}_h, \qquad
&\vec{t}:=\vmag+\vmag_h.
\end{aligned}\label{eq:st-notation}
\end{equation}
Finally, we use the notation $\stackrel{(\cdot)}{=}$ and $\stackrel{(\cdot)}{\le}$ to denote
that the equality or inequality is justified by equation $(\cdot)$.

\subsection{Derivation of the weak form of the adjoint}
\label{sec:adj_deriv_details}
In this section we provide derivation for the  primal linearized operators
$\ofancy{J}_{21}^*=\ovfancy{Y}^*$, $\ofancy{J}_{11}^*=\ovfancy{Z}_{\vec{u}}^*$,
$\ofancy{J}_{12}^*=\ovfancy{Z}_{\vmag}^*$ and $\ofancy{J}_{31}^*=\vfancy{C}^*$
in  \eqref{eq:averaged-entries}.  We first compute the  primal linearized
operators, $\ovfancy{Y} = \ofancy{J}_{21}$, $\ovfancy{Z}_{\vec{u}} =
\ofancy{J}_{11}$,  $\ovfancy{Z}_{\vmag}=\ofancy{J}_{12}$ and
$\vfancy{C}=\ofancy{J}_{31}$, using \eqref{eq:averaged-jacobian} and then apply
\eqref{eq:basic_adoint_id} to compute the $\ofancy{J}_{ij}^*$s. We have from
\eqref{eq:averaged-jacobian} for  $\vmtestD\in\vec{H}_\tau^1(\Omega)$ and
$\vec{w}\in\vec{H}_0^1(\Omega)$,
\begin{align*}
\ovfancy{Y}\,\vmtestD&:=\int_0^1\frac{\p\vfancy{Y}}{\p\vmag}(s\vmag + (1-s)\vmag_h)\vmtestD\d s,\\
\ovfancy{Z}_{\vmag}\,\vmtestD&:=\int_0^1\frac{\p\vfancy{Z}}{\p\vmag}(s\vec{u} + (1-s)\vec{u}_h)\vmtestD\d s,\\
\ovfancy{Z}_{\vec{u}}\,\vec{w}&:=\int_0^1\frac{\p\vfancy{Z}}{\p\vec{u}}(s\vmag + (1-s)\vmag_h)\vec{w}\d s.
\end{align*}
To this end, we compute
\begin{equation}
\begin{aligned}
&\ovfancy{Y}\,\vmtestD=\int_0^1\frac{\p\vfancy{Y}}{\p\vmag}(s\vmag + (1-s)\vmag_h)\vmtestD\d s\\
&= \int_0^1\left[\curl(s\vmag+(1-s)\vmag_h)\right]\times\vmtestD + (\curl\vmtestD)\times(s\vmag+(1-s)\vmag_h)\d s\\
&=\frac{1}{2}\left[(\curl(\vmag_h+\vmag))\times\vmtestD+(\curl\vmtestD)\times(\vmag_h+\vmag)\right].
\end{aligned}\label{eq:ybar}
\end{equation}
Similarly, for the two $\ovfancy{Z}$ terms,
\begin{equation}
\begin{aligned}
&\ovfancy{Z}_{\vmag}\,\vmtestD=\int_0^1\frac{\p\vfancy{Z}}{\p\vmag}(s\vec{u} + (1-s)\vec{u}_h)\vmtestD\d s\\
&=\int_0^1\curl((s\vec{u}+ (1-s)\vec{u}_h)\times\vmtestD)\d s
=\frac{1}{2}\left[\curl((\vec{u}_h+\vec{u})\times\vmtestD)\right].
\end{aligned}\label{eq:Zu-bar}
\end{equation}
An identical procedure produces,
\begin{equation}
\ovfancy{Z}_{\vec{u}}\,\vec{w}=\frac{1}{2}\left[\curl(\vec{w}\times(\vmag+\vmag_h))\right].
\end{equation}
Now, to find the adjoints of these operators, we use \eqref{eq:basic_adoint_id}, which in our case involves multiplying by a test function and
then isolating the trial function using integration by parts. 
We  also make use of the vector identities in Appendix
\ref{append:vect-ind}.  

We are now prepared to compute the adjoint for
$\ovfancy{Y}$. 
Integrating \eqref{eq:ybar} 
against $\vec{v}\in\vec{H}_0^1(\Omega)$, 
\begin{align*}
&(\ovfancy{Y}\,\vmtestD,\vec{v})
=\frac{1}{2}\int_\Omega\left[(\curl\vec{t})\times\vmtestD+(\curl\vmtestD)\times\vec{t}\right]\cdot\vec{v}\d x\\
&\stackrel{\eqref{piped}}{=}\frac{1}{2}\int_\Omega\vmtestD\cdot\left[\vec{v}\times(\curl\vec{t})\right]
+(\curl\vmtestD)\cdot\left[\vec{t}\times\vec{v}\right]\d x\\
&\stackrel{\eqref{div-cross-int}}{=}\frac{1}{2}\int_\Omega-\vmtestD\cdot\left[(\curl\vec{t})\times\vec{v}\right]
+ \vmtestD\cdot\left[\curl(\vec{t}\times\vec{v})\right]\d x
- \frac{1}{2}\int_{\p\Omega}\vmtestD\cdot\left[(\vec{t}\times\vec{v})\times\vec{n}\right]\d s\\
&\stackrel{\eqref{piped}}{=}\frac{1}{2}\int_\Omega-\vmtestD\cdot\left[(\curl\vec{t})\times\vec{v}\right]
+ \vmtestD\cdot\left[\curl(\vec{t}\times\vec{v})\right]\d x
+ \frac{1}{2}\int_{\p\Omega}(\vec{t}\times\vec{v})\cdot\left[\vmtestD\times\vec{n}\right]\d s \\
&\stackrel{\eqref{eq:H-tau-one}}{=} \frac{1}{2}\int_\Omega-\vmtestD\cdot\left[(\curl\vec{t})\times\vec{v}\right]
+ \vmtestD\cdot\left[\curl(\vec{t}\times\vec{v})\right]\d x
\stackrel{\eqref{eq:averaged-entries}}{=}(\vmtestD,\ovfancy{Y}^*\vec{v}).
\end{align*}
We proceed with computing the adjoint for $\ovfancy{Z}_{\vec{u}}$, with
$\vec{c}\in\vec{H}_\tau^1(\Omega)$,
\begin{align*}
&(\ovfancy{Z}_{\vec{u}}\,\vec{w},\vmtestC) = \frac{1}{2}\left(\curl(\vec{w}\times\vec{t}),\vmtestC\right)\\
&\stackrel{\eqref{div-cross-int}}{=}\frac{1}{2}\int_\Omega(\vec{w}\times\vec{t})\cdot(\curl\vmtestC)\d x
- \frac{1}{2}\int_{\p\Omega}(\vec{w}\times\vec{t})\cdot(\vmtestC\times\vec{n})\d s\\
&\stackrel{\eqref{piped}}{=}\frac{1}{2}\int_\Omega\vec{w}\cdot\left[\vec{t}\times(\curl\vmtestC)\right]\d x
- \frac{1}{2}\int_{\p\Omega}(\vec{w}\times\vec{t})\cdot(\vmtestC\times\vec{n})\d s\\
&\stackrel{\eqref{eq:H-tau-one}}{=}\frac{1}{2}\int_\Omega\vec{w}\cdot\left[\vec{t}\times(\curl\vmtestC)\right]\d x
\stackrel{\eqref{eq:averaged-entries}}{=}
(\vec{w},\ovfancy{Z}_{\vec{u}}^*\,\vmtestC).
\end{align*}
Finally we compute the adjoint to the linearized operator
$\ovfancy{Z}_{\vmag}$, again with $\vec{c}\in\vec{H}_\tau^1(\Omega)$,
\begin{align*}
&(\ovfancy{Z}_{\vmag}\,\vmtestD,\vmtestC)=\frac{1}{2}\left(\curl(\vec{s}\times\vmtestD),\vmtestC\right)\\
&\stackrel{\eqref{div-cross-int}}{=}\frac{1}{2}\int_\Omega(\vec{s}\times\vmtestD)\cdot(\curl\vmtestC)\d x
- \frac{1}{2}\int_{\p\Omega}(\vec{s}\times\vmtestD)\cdot(\vmtestC\times\vec{n})\d s\\
&\stackrel{\eqref{piped}}{=} \frac{1}{2}\int_\Omega\vmtestD\cdot\left[(\curl\vmtestC)\times\vec{s}\right]\d x
- \frac{1}{2}\int_{\p\Omega}\vmtestD\cdot\left[\vec{s}\times(\vmtestC\times\vec{n})\right] 
- (\vec{s}\times\vmtestD)\cdot(\vmtestC\times\vec{n})\d s\\
&\stackrel{\eqref{eq:H-tau-one}}{=}\frac{1}{2}\int_\Omega\vmtestD\cdot\left[(\curl\vmtestC)\times\vec{s}\right]\d x  \stackrel{\eqref{eq:averaged-entries}}{=} (\vmtestD,\ovfancy{Z}_{\vmag}^*\,\vmtestC).
\end{align*}
The operator $\vfancy{C}^*$  is identical to the one presented in \cite{Estep2010}.

\subsection{Well posedness of the adjoint problem}
\label{sec:weak_form_adj_well_posedness}
In this section we prove the well-posedness of the adjoint problem
\S\ref{sect:MHD-weak-adjoint} equation \eqref{eq:EP-dual-problem} using a
 saddle point type argument.  To keep consistent with the standard
setting of saddle point problems \cite{ern_guermond_2011, brenner_scott_2011},
we use the notation $X:=\vec{H}^1_0(\Omega)\times\vec{H}^1_\tau(\Omega)$ and $M:=L^2(\Omega)$
so that $\EPspace = X\times M$.
We equip the space $X$ with the graph norm
\begin{equation}
\|(\vec{v},\vmtestC)\|_X := (\|\vec{v}\|_1^2 + \|\vmtestC\|_{1}^2)^{1/2}.
\end{equation}

We next
define the bilinear form $a:X\times X\to\R$ by
\begin{equation}
\begin{aligned}
  a((\vphi, \vbeta), (\vec{v},
  \vmtestC))=\frac{1}{\Re}(\nabla\vphi,\nabla\vec{v})+
  \left(\ovfancy{C}^*\vphi, \vec{v}\right)\\
  +\frac{\kappa}{\Rm}\left(\curl\vbeta,\curl\vmtestC\right)+
  \frac{\kappa}{\Rm}\left(\diver\vbeta,\diver\vmtestC\right)\\
  -\kappa\left(\ovfancy{Y}^*\vphi,\vmtestC\right)
-\kappa\left(\ovfancy{Z}_{\vec{u}}^*\vbeta,\vec{v}\right)-
\kappa\left(\ovfancy{Z}_{\vmag}^*\vbeta,\vmtestC\right),
\end{aligned}
\end{equation}
and the mixed form $b:X\times M\to\R$ by
\begin{equation}
b((\vphi, \vmtestC), \pi)=(\pi,\diver \vphi).
\end{equation}
The weak dual problem \eqref{eq:EP-dual-problem} is then equivalent to the following mixed
problem: find
$((\vphi, \vbeta), \pi)\in X\times M$ such that
\begin{equation}
\begin{cases}
a((\vphi, \vbeta), (\vec{v}, \vmtestC)) + b((\vec{v}, \vmtestC), \pi)
= f(\vec{v},\vmtestC),\,&\forall(\vec{v},\vmtestC)\in X,\\
b((\vphi, \vbeta), q) = -g(q),\,&\forall q\in M,
\end{cases}\label{eq:saddle-form}
\end{equation}
where $f(\vec{v},\vmtestC)=(\vec{\psi}_{\vec{u}},\vec{v})
+ (\vec{\psi}_{\vmag}, \vec{c})$, $g(q) = (\psi_p, q)$  and $\Psi
= \begin{bmatrix}\vec{\psi}_{\vec{u}},\vec{\psi}_{\vmag},\psi_p\end{bmatrix}^T$
so that  $(\Psi, V) = f(\vec{v},\vmtestC) + g(q)$.  According to the theory of
saddle point systems, in order to show the existence and uniqueness of
solutions to \eqref{eq:saddle-form}, it suffices to show:
\begin{enumerate}[(i)]
\item The bilinear forms $a(\cdot,\cdot)$ and $b(\cdot,\cdot)$ are bounded on their respective domains.
\item The form $a(\cdot,\cdot)$ is coercive on $X_0:=\{v\in X:b(v,q) = 0,\,\forall q\in M\}$.
\item The form $b(\cdot,\cdot)$ satisfies the  inf-sup condition: $\exists\beta >0$ such that
\begin{equation}
\inf_{q\in M}\sup_{(\vec{v},\vmtestC)\in X}\frac{b((\vec{v},\vmtestC),q)}{\|(\vec{v},\vmtestC)\|_X\|q\|_M}\ge \beta.
\end{equation}
\end{enumerate}
We organize these parts in the following lemmas. We make frequent use of
the inequalities in Appendix \ref{app:useful_ineqs} in the proofs.

\begin{lemma}\label{lem:boundedness}
The form $a(\cdot,\cdot)$ is bounded on $X$.
\end{lemma}
\begin{proof}
Consider the splitting 
\begin{equation}
a((\vphi,\vbeta),(\vec{v},\vmtestC))=a_0((\vphi,\vbeta),(\vec{v},\vmtestC))
+ a_1((\vphi, \vbeta), (\vec{v}, \vmtestC))\label{eq:splitting}
\end{equation}
where
\begin{align*}
&a_0((\vphi, \vbeta), (\vec{v}, \vmtestC))=\frac{1}{\Re}(\nabla\vphi,\nabla\vec{v})
+\frac{\kappa}{\Rm}\left(\curl\vbeta,\curl\vmtestC\right)
+ \frac{\kappa}{\Rm}\left(\diver\vbeta,\diver\vmtestC\right),\\
&a_1((\vphi, \vbeta), (\vec{v}, \vmtestC))=\left(\ovfancy{C}^*\vphi, \vec{v}\right)
-\kappa\left(\ovfancy{Y}^*\vphi,\vmtestC\right)
-\kappa\left(\ovfancy{Z}_{\vec{u}}^*\vbeta,\vec{v}\right)-\kappa\left(\ovfancy{Z}_{\vmag}^*\vbeta,\vmtestC\right).
\end{align*}
Then it suffices to show that both $a_0(\cdot,\cdot)$ and $a_1(\cdot,\cdot)$
are bounded separately. The proof for the boundedness of  $a_0$ is given in
\cite{gunzburger_meir_peterson_1991}. For $a_1$ observe that
\begin{equation}
\begin{aligned}
|a_1((\vphi, \vbeta), (\vec{v}, \vmtestC))|\le\int_\Omega\left|\ovfancy{C}^*\vphi\cdot\vec{v}\right|\d x
+ \kappa\int_\Omega\left|\ovfancy{Y}^*\vphi\cdot\vmtestC\right|\d x\\
+ \kappa\int_\Omega\left|\ovfancy{Z}^*_{\vec{u}}\vbeta\cdot\vec{v}\right|\d x
+ \kappa\int_\Omega\left|\ovfancy{Z}^*_{\vmag}\vbeta\cdot\vmtestC\right|\d x.
\end{aligned}\label{eq:init-bound}
\end{equation}

Now, for the first term on the right hand side of \eqref{eq:init-bound},
\begin{align*}
&\int_\Omega\left|\ovfancy{C}^*\vphi\cdot\vec{v}\right|\d x=\frac{1}{2}\int_\Omega\big|\left[(\nabla\vec{s})^T\vphi-\left((\vec{s}\cdot\nabla)\vphi\right)
-(\diver\vec{s})\vphi\right]\cdot\vec{v}\big|\d x\\
&=\frac{1}{2}\int_\Omega\big|\vphi^T(\nabla\vec{s})\vec{v}-\vec{v}^T(\nabla\vphi)\vec{s}-(\diver\vec{s})(\vphi\cdot\vec{v})\big|\d x\\
&\stackrel{\eqref{eq:quad-form-bound}}{\le}\frac{1}{2}\left[\|\vphi\|_{\vec{L}^4} \|\vec{s}\|_1\|\vec{v}\|_{\vec{L}^4}
+\|\vphi\|_1\|\vec{s}\|_{\vec{L}^4}\|\vec{v}\|_{\vec{L}^4} 
+ \|\diver\vec{s}\|\|\vphi\cdot\vec{v}\|\right]\\
&\stackrel{\eqref{eq:div-ineq}}{\le}\frac{1}{2}\left[\|\vphi\|_{\vec{L}^4} \|\vec{s}\|_1\|\vec{v}\|_{\vec{L}^4}
+\|\vphi\|_1\|\vec{s}\|_{\vec{L}^4}\|\vec{v}\|_{\vec{L}^4} + \sqrt{3}\|\vec{s}\|_1\|\vphi\|_{\vec{L}^4}\|\vec{v}\|_{\vec{L}^4}\right]\\
&\stackrel{\eqref{eq:ineq_embedding}}{\leq} \frac{\embed}{2} \left ( \|\vphi \|_{1} \| \vec{s}\|_{1}  \|\vec{v}\|_{1}  
+ \|\vec{s} \|_{1} \| \vphi\|_{1} \|\vec{v} \|_{1} 
+ \sqrt{3} \|\vec{s}\|_1 \| \vphi \|_{1} \| \vec{v} \|_{1}
    \right)\\
&\leq \frac{3\sqrt{3} \embed  }{2}  \|\vec{s}\|_1 \| \vphi \|_{1} \| \vec{v} \|_{1},
\end{align*}
where $\embed$ is the square of the embedding constant of $\vec{H}^1(\Omega)$
into $\vec{L}^4(\Omega)$, see \eqref{eq:ineq_embedding}.
For the second term on the right hand side of \eqref{eq:init-bound},
\begin{align*}
    &\kappa\left(\ovfancy{Y}^*\vphi\cdot\vmtestC\right) \le \frac{\kappa}{2}\int_\Omega\big|\vmtestC\cdot\left[(\curl\vec{t})\times\vphi\right]\big|
    + \big|\vmtestC\cdot\left[\curl(\vec{t}\times\vphi)\right]\big|\d x\\
    &\stackrel{\eqref{div-cross-int}}{=}\frac{\kappa}{2}\int_\Omega\big|\vmtestC\cdot\left((\curl\vec{t})\times\vphi\right)\big|
    + \big|(\curl\vmtestC)\cdot\left(\vec{t}\times\vphi\right)\big|\d x\\
    &\stackrel{\eqref{piped}}{=}\frac{\kappa}{2}\int_\Omega\big|(\curl\vec{t})\cdot\left(\vmtestC\times\vphi\right)\big|
    + \big|(\curl\vmtestC)\cdot\left(\vec{t}\times\vphi\right)\big|\d x\\
    &\stackrel{\eqref{eq:cross-ineq}}{\le} \frac{\kappa}{2}\left(\|\curl\vec{t}\|_{\vec{L}^2}\|\vmtestC\|_{\vec{L}^4}\|\vphi\|_{\vec{L}^4}
    + \|\curl\vmtestC\|_{\vec{L}^2}\|\vec{t}\|_{\vec{L}^4}\|\vphi\|_{\vec{L}^4}\right)\\
    &\stackrel{\eqref{eq:curl-ineq}}{\le} \frac{\kappa\sqrt{2}}{2}\left(\|\vmtestC\|_{\vec{L}^4}\|\vec{t}\|_{1}\|\vphi\|_{\vec{L}^4}
    + \|\vmtestC\|_{1}\|\vec{t}\|_{\vec{L}^4}\|\vphi\|_{\vec{L}^4}\right)\\
		&\stackrel{\eqref{eq:ineq_embedding}}{\le} \kappa\embed \sqrt{2}\|\vmtestC\|_{1}\|\vec{t}\|_{1}\|\vphi\|_{1}.
\end{align*}
	
For the third term on the right hand side of \eqref{eq:init-bound},
\begin{align*}
\kappa\left(\ovfancy{Z}_{\vec{u}}^*\vbeta,\vec{v}\right)
&\le\frac{\kappa}{2}\int_\Omega\big|\vec{v}\cdot\left[\vec{t}\times(\curl\vec{\vbeta})\right]\big|\d x
\stackrel{\eqref{div-cross-int}}{=}\frac{\kappa}{2}\int_\Omega\big|(\vec{v}\times\vec{t})\cdot(\curl\vbeta)\big|\d x\\
&\stackrel{\eqref{eq:curl-ineq}}{\le}\frac{\kappa \sqrt{2}}{2}\|\vec{v}\|_{\vec{L}^4}\|\vec{t}\|_{\vec{L}^4}\|\vbeta\|_{1}
\stackrel{\eqref{eq:ineq_embedding}}{\le}\frac{\kappa \embed \sqrt{2}}{2}\|\vec{v}\|_{1}\|\vec{t}\|_{1}\|\vbeta\|_{1}.
\end{align*}
The fourth term follows the same argument as the third term to yield the bound,
\begin{align}
\kappa\left(\ovfancy{Z}_{\vmag}^*\vbeta,\vmtestC\right)
\le\frac{\kappa \embed \sqrt{2}}{2}\|\vmtestC\|_{1}\|\vec{s}\|_{1}\|\vbeta\|_{1}.
\end{align}
Putting these bounds together,
we conclude
\begin{equation}
\begin{aligned}
a_1((\vphi, \vbeta), (\vec{v},\vmtestC)) \le\embed \bigg(\frac{3\sqrt{3}}{2}  \|\vec{s}\|_1 \| \vphi \|_{1} \| \vec{v} \|_{1}
+\kappa\sqrt{2}\|\vmtestC\|_{1}\|\vec{t}\|_{1}\|\vphi\|_{1}\\
+\frac{\kappa \sqrt{2}}{2}\|\vec{v}\|_{1}\|\vec{t}\|_{1}\|\vbeta\|_{1}
+\frac{\kappa \sqrt{2}}{2}\|\vmtestC\|_{1}\|\vec{s}\|_{1}\|\vbeta\|_{1}\bigg)\\
\stackrel{\eqref{eq:CS}}{\le}\embed \Bigg(\frac{3\sqrt{3}}{2}  \|\vec{s}\|_1 \| \vphi \|_{1} \| \vec{v} \|_{1}
+\frac{\kappa \sqrt{2}}{2}\|\vmtestC\|_{1}\|\vec{s}\|_{1}\|\vbeta\|_{1}\\
+\|\vec{t}\|_1\kappa\sqrt{2}\|(\vec{v},\vmtestC)\|_{X}\|(\vphi,\vbeta)\|_{X}\Bigg)\\
\stackrel{\eqref{eq:CS}}{\le}
\embed \Bigg(\|\vec{s}\|_1\max\left\{\frac{3\sqrt{3}}{2},\frac{\kappa\sqrt{2}}{2}\right\}\|(\vec{v},\vmtestC)\|_{X}\|(\vphi,\vbeta)\|_{X}\\
+\|\vec{t}\|_1\|(\vec{v},\vmtestC)\|_{X}\|(\vphi,\vbeta)\|_{X}\Bigg)\\
\le \alpha_b\|(\vec{v},\vmtestC)\|_{X}\|(\vphi,\vbeta)\|_{X},
\end{aligned}\label{eq:final-continuity}
\end{equation}
where 
\[\alpha_b = \max\left\{\|\vec{s}\|_1\max\left\{\frac{3\sqrt{3}}{2},\frac{\kappa\sqrt{2}}{2}\right\}, \|\vec{t}\|_1\right\}.\]
\end{proof}
Now we consider the coercivity of the bilinear form $a(\cdot, \cdot)$ on
$X$.
\begin{lemma}\label{lem:coerc}
There exists a constant $\alpha_c > 0$ such that whenever
\begin{equation}
\frac{k_1}{\Re}-\embed \left[\frac{3\sqrt{3}
}{2}\|\vec{s}\|_1+\frac{3\kappa\sqrt{2}}{4}\|\vec{t}\|_1\right]>0,\label{eq:pos1}
\end{equation}
and
\begin{equation}
\frac{k_2\kappa}{\Rm^2}-\embed
\left[\frac{\kappa\sqrt{2}}{2}\|\vec{s}\|_1+\frac{3\kappa\sqrt{2}}{4}\|\vec{t}\|_1\right]>0\label{eq:pos2}
\end{equation}
then
\begin{equation}
a((\vphi, \vbeta),(\vphi,\vbeta)) \ge\alpha_c\|(\vphi,\vbeta)\|_{X}^2,\quad\forall (\vphi,\vbeta)\in X.
\end{equation}
\end{lemma}
\begin{proof}
Using the splitting established in the previous lemma,
\begin{equation}
\begin{aligned}
&a((\vphi,\vbeta),(\vphi, \vbeta))\ge  a_0((\vphi,\vbeta),(\vphi, \vbeta))\ -\left|a_1((\vphi, \vbeta), (\vphi, \vbeta))\right|\\
&=\frac{1}{\Re}(\nabla\vphi,\nabla\vphi)
+\frac{\kappa}{\Rm}\left(\curl\vbeta,\curl\vbeta\right)
+ \frac{\kappa}{\Rm}\left(\diver\vbeta,\diver\vbeta\right)\\
&-\left|a_1((\vphi, \vbeta), (\vphi, \vbeta))\right|\\
&\ge\frac{k_1}{\Re}\|\vphi\|_1^2+\frac{k_2\kappa}{\Rm^2}\|\vbeta\|_1^2
-\left|a_1((\vphi, \vbeta), (\vphi, \vbeta))\right|
\end{aligned}
\label{eq:first-coerc}
\end{equation}
where $k_1$ comes from the Poincar\'e type inequality \eqref{eq:poincare},
and $k_2$ is defined though
\begin{equation}
\|\curl\vec{v}\|_0^2 + \|\diver\vec{v}\|_0^2\ge k_2\|\vec{v}\|_1^2,\quad\forall\vec{v}\in\vec{H}_{\tau}^1(\Omega),
\end{equation}
which is valid under the restrictions we have imposed on the domain $\Omega$
and the continuous embedding of
$\vec{H}_\tau^1(\Omega)\hookrightarrow\vec{H}^1(\Omega)$~\cite{Girault1986,gunzburger_meir_peterson_1991}.
Picking up from \eqref{eq:first-coerc} and using \eqref{eq:pythag-ineq} we
conclude that,
\begin{align*}
&a((\vphi,\vbeta),(\vphi, \vbeta))\ge
\frac{k_1}{\Re}\|\vphi\|_1^2+\frac{k_2\kappa}{\Rm^2}\|\vbeta\|_1^2 - |a_1((\vphi,\vbeta),(\vphi, \vbeta))|\\
&\stackrel{\eqref{eq:final-continuity}}{\ge}
\left(\frac{k_1}{\Re}-\frac{\embed 3\sqrt{3} }{2}\|\vec{s}\|_1\right) \|\vphi\|_1^2
+\left(\frac{k_2\kappa}{\Rm^2}-\frac{\embed \kappa\sqrt{2}}{2}\|\vec{s}\|_1\right)\|\vbeta\|_1^2\\
&- \frac{\embed 3\kappa\sqrt{2}}{2}\|\vphi\|_1\|\vec{t}\|_1\|\vbeta\|_1\\
&\stackrel{\eqref{eq:pythag-ineq}}{\ge}
\left(\frac{k_1}{\Re}-\frac{\embed 3\sqrt{3} }{2}\|\vec{s}\|_1\right) \|\vphi\|_1^2
+\left(\frac{k_2\kappa}{\Rm^2}-\frac{\embed \kappa\sqrt{2}}{2}\|\vec{s}\|_1\right)\|\vbeta\|_1^2\\
&-\frac{\embed 3\kappa\sqrt{2}}{4}\|\vec{t}\|_1\left(\|\vbeta\|_1^2+\|\vphi\|_1^2\right)\\
&=\left(\frac{k_1}{\Re}-\embed \left[\frac{3\sqrt{3} }{2}\|\vec{s}\|_1+\frac{3\kappa\sqrt{2}}{4}\|\vec{t}\|_1\right]\right) \|\vphi\|_1^2\\
&+\left(\frac{k_2\kappa}{\Rm^2}-\embed \left[\frac{\kappa\sqrt{2}}{2}\|\vec{s}\|_1+\frac{3\kappa\sqrt{2}}{4}\|\vec{t}\|_1\right]\right)\|\vbeta\|_1^2.
\end{align*}
Thus, taking
\begin{equation}
\begin{aligned}
\alpha_c=\min\Bigg\{\frac{k_1}{\Re}-\embed \left[\frac{3\sqrt{3} }{2}\|\vec{s}\|_1+\frac{3\kappa\sqrt{2}}{4}\|\vec{t}\|_1\right],\\
\frac{k_2\kappa}{\Rm^2}-\embed \left[\frac{\kappa\sqrt{2}}{2}\|\vec{s}\|_1+\frac{3\kappa\sqrt{2}}{4}\|\vec{t}\|_1\right]\Bigg\},
\end{aligned}
\end{equation}
concludes the lemma.
\end{proof}
\begin{remark}
We note that the quantities assumed to be positive  in \eqref{eq:pos1} and
\eqref{eq:pos2}, depend on the computed and true solutions through
$\|\vec{s}\|$ and $\|\vec{t}\|$, which should should both be bounded for ``small
data'' as described precisely in Theorem 4.7 of
\cite{gunzburger_meir_peterson_1991}. 
The two quantities in \eqref{eq:pos1} and \eqref{eq:pos2} also depend on the
fluid and magnetic Reynolds numbers ($\Re$ and $\Rm$ respectively).  In
particular, for small to moderate $\Re$ and $\Rm$ these inequalities might very
well be satisfied, which is the case for dissipative MHD. However, the larger
are $\Re$ and $\Rm$ (and in particular for the limit as $\Re,\Rm\to\infty$,
that is in the case of ideal MHD), the smaller the positive terms of
\eqref{eq:pos1} and \eqref{eq:pos2}, and thus coercivity cannot be proven by
this method.  We conclude this method might therefore need to be adapted for
high $\Re$ or $\Rm$ flows to guarantee coercivity.
\end{remark}

Now we are prepared to prove the main result.
\begin{theorem}
Under the conditions of Lemma \ref{lem:coerc} there exists a unique solution to the dual problem \eqref{eq:EP-dual-problem}.
\end{theorem}
\begin{proof}
The boundedness and inf-sup condition for $b(\cdot,\cdot)$ are standard see e.g. \cite{brenner_scott_2011}.
The boundedness of $a(\cdot,\cdot)$ follows from Lemma \ref{lem:boundedness}, and Lemma \ref{lem:coerc} proves
$a(\cdot,\cdot)$ is coercive on $X$ so in particular on $X_0$.
\end{proof}

\section{Conclusions}
We have presented an adjoint-based \textit{a posteriori} analysis of adjoint
for an exact penalty formulation of incompressible resistive MHD.  This
included the derivation of the adjoint error estimate, and a development that
characterized the separate contributions of error from the momentum, continuity
and magnetic field equations. The numerical examples illustrated both the
accuracy  as well as the usefulness of the error estimate
for the the assessment of the respective sources of the error from the
different  physics components. The example QoIs included two differing physically
meaningful  quantities, the averaged velocity-related to the flow rate, and
the induced magnetic field strength.

The novel aspects of this work include defining an adjoint problem for an
overdetermined system, namely the stationary MHD equations. In particular, the
standard definition of an adjoint operator does not suffice and we must define
the adjoint directly for the weak problem.  Moreover, we prove the
well-posedness of the adjoint problem.  The error estimates derived in this
article are also amenable for using in adaptive refinement algorithms e.g. see
\cite{AO2000,CEJLT10,ABC+2011,CET+2016,Fidkowski_Darmofal_2011,chaudhry_2018}.

\appendix
\section{Standard function spaces}
\label{app:func_spaces}
 We denote by $L^2(\Omega)$ the set of all square Lebesgue
integrable functions on $\Omega \subset \mathbb{R}^d$ with associated inner product
$(\cdot,\cdot)$ and norm $\|\cdot\|$. This extends
naturally to vector valued functions, denoted by $\vec{L}^2(\Omega)$, where the inner product is given by,
\begin{equation*}
(\vec{u},\vec{v}) = \sum_{i=1}^d(u_i, v_i).
\end{equation*}

The Sobolev norm for $p=2$ is,
\begin{equation*}
	\|v\|_m:=\left(\sum_{|\alpha|=0}^m\big\|D^\alpha v\big\|^2\right)^{1/2}.
\end{equation*}
where $\alpha=(\alpha_1,\dots,\alpha_m)$ is a multi-index of length $m$ and
\[D^\alpha v:=\p_{x_1}^{\alpha_1}\dots\p_{x_m}^{\alpha_m}v,\]
where the partial derivatives are taken in the weak sense.
Thus, the Hilbert spaces $H^m$ for $m=0,1,2,\dots$ is simply
be defined as functions with bounded $m$-norm,
\begin{equation*}
	H^m(\Omega):=\{v:\|v\|_m<\infty\}.
\end{equation*}
The space $H^0(\Omega)$ is identified with $L^2(\Omega)$.
For vector valued functions, the Hilbert space $\vec{H}^m$ is defined as,
\begin{equation*}
	\vec{H}^m(\Omega):=\{\vec{v}:v_i\in H^m(\Omega),\,i=1,\dots,d\},
\end{equation*}
with associated norm
\begin{equation*}
\|\vec{v}\|_m=\left(\sum_{i=1}^d \|v_i\|_{m}^2\right)^{1/2}.
\end{equation*}
\section{Vector identities and inequalities}\label{append:vect-ind}
We use the following vector identities,
\begin{subequations}
\begin{align}
	\vec{A}\cdot(\vec{B}\times \vec{C}) &= \vec{B}\cdot(\vec{C}\times \vec{A})=\vec{C}\cdot(\vec{A}\times\vec{B}),\label{piped}\\
		\int_\Omega \vec{A}\cdot (\curl \vec{B})\d x &= \int_\Omega \vec{B}\cdot(\curl \vec{A})\d x
  -\int_{\p\Omega}\vec{B}\cdot (\vec{A}\times \vec{n})\d s.\label{div-cross-int}
\end{align}
\end{subequations}
We also make use of the following inequalities for $\vec{u},\vec{v}\in\vec{H}^1(\Omega)$,
\begin{subequations}
\begin{align}
|\vec{u}\cdot\vec{v}|&\le\|\vec{u}\|_{\R^d}\|\vec{v}\|_{\R^d},\\
\|\vec{u}\times\vec{v}\|_{\R^d}&\le\|\vec{u}\|_{\R^d}\|\vec{v}\|_{\R^d},\label{eq:cross-ineq}\\
\|\curl\vec{u}\|_{\R^d}&\le\sqrt{2}\|\nabla\vec{u}\|_{\R^{d\times d}},\label{eq:curl-ineq}\\
|\diver\vec{u}|&\le\sqrt{3}\|\nabla\vec{u}\|_{\R^{d\times d}}\label{eq:div-ineq}\\
\|A \vec{v}\|_{\R^d} & \le \|A\|_{\R^{d\times d}}\|\vec{v}\|_{\R^d}\label{eq:mat-ineq},
\end{align}\label{eq:append-inequalities}
\end{subequations}
and finally the equality
\begin{equation}
\|\nabla\vec{v}^T\|_{\R^{d\times  d}} = \| \nabla\vec{v}\|_{\R^{d\times  d}},\label{eq:grad-transpose}
\end{equation}

\section{Useful inequalities from analysis}
\label{app:useful_ineqs}
\begin{enumerate}
\item The space $\vec{H}^1(\Omega)$ embeds continuously in $\vec{L}^4(\Omega)$ with constant $\sqrt{\embed}$. That is,
$\vec{H}^1(\Omega)\hookrightarrow\vec{L}^4(\Omega)$ such that,
\begin{equation}
\|\vec{v}\|_{\vec{L}^4} \le \sqrt{\embed}\|\vec{v}\|_{\vec{H}^1}. \label{eq:ineq_embedding}
\end{equation}
\item The Cauchy-Schwarz inequality for $\begin{bmatrix}a,b\end{bmatrix},\begin{bmatrix}c,d\end{bmatrix}\in\R^2$,
\begin{equation}
ac + bd = \begin{bmatrix}a, b\end{bmatrix} \begin{bmatrix}c,d\end{bmatrix}^T \le \sqrt{a^2+c^2}\sqrt{b^2 + d^2},\label{eq:CS}
\end{equation}

\item The following inequality follows from the Poincar\'e inequality,
\begin{equation}
\label{eq:poincare}
\|\nabla\vec{v}\|_0^2\ge k_1\|\vec{v}\|_1^2,\quad\forall\vec{v}\in\vec{H}_0^1(\Omega).
\end{equation}

\item For $x,y\in\R$,
\begin{equation}
-xy \ge -\tfrac{1}{2}(x^2+y^2),\label{eq:pythag-ineq}
\end{equation}
\end{enumerate}
We also need the following propositions,
\begin{prop}
Let $\vec{u},\vec{v}, \vec{w}\in\vec{H}^1(\Omega)$. Then there holds
\begin{equation}
\int_\Omega\vec{u}^T(\nabla\vec{v})\vec{w}\d x\le \|\vec{u}\|_{\vec{L}^4}\|\vec{w}\|_{\vec{L}^4}\|\vec{v}\|_1.\label{eq:quad-form-bound}
\end{equation}
\end{prop}
\begin{proof}
We will work with the integrand first. To this end, we have that
\begin{align*}
\vec{u}^T(\nabla\vec{v})\vec{w}
=
\sum_{i=1}^d u_i \vec{w}^T\nabla v_i
\le \sum_{i=1}^d|u_i|\|\vec{w}\|_{\R^d}\|\nabla v_i\|_{\R^d}
= \|\vec{w}\|_{\R^d}\sum_{i=1}^d|u_i|\|\nabla v_i\|_{\R^d}\\
\le \|\vec{w}\|_{\R^d}\left(\sum_{i=1}^d|u_i|^2\right)^{1/2}\left(\sum_{i=1}^d\|\nabla v_i\|_{\R^d}^2\right)^{1/2}
=
 \|\vec{w}\|_{\R^d}\|\vec{u}\|_{\R^d}\|\nabla\vec{v}\|_{\R^{d\times d}}.
\end{align*}
Now we integrate,
\begin{align*}
&\int_\Omega|\vec{w}\|_{\R^d}\|\vec{u}\|_{\R^d}\|\nabla\vec{v}\|_{\R^{d\times d}}\d x\\
&\le\left(\int_\Omega\|\vec{u}\|_{\R^d}^2\|\vec{w}\|_{\R^d}^2\d x\right)^{1/2}\left(\int_\Omega\|\nabla\vec{v}\|_{\R^{d\times d}}^2\right)^{1/2}\\
&\le\left(\int_\Omega\|\vec{u}\|_{\R^d}^4\d x\right)^{1/4}\left(\int_\Omega\|\vec{w}\|_{\R^d}^4\d x\right)^{1/4}\left(\int_\Omega\|\nabla\vec{v}\|_{\R^{d\times d}}^2\d x\right)^{1/2}\\
&=\|\vec{u}\|_{\vec{L}^4}\|\vec{w}\|_{\vec{L}^4}|\vec{v}|_1\le\|\vec{u}\|_{\vec{L}^4}\|\vec{w}\|_{\vec{L}^4}\|\vec{v}\|_1.
\end{align*}
\end{proof}

\bibliographystyle{siamplain}
\bibliography{analysis_article}
\end{document}